\documentclass[11pt]{article}

\usepackage[margin=1in]{geometry}
\usepackage[dvipsnames]{xcolor}
\usepackage{mathpazo}
\usepackage{booktabs}
\usepackage{stmaryrd}
\usepackage[english]{babel}
\usepackage[autostyle, english = american]{csquotes}
\usepackage{anyfontsize}
\usepackage{authblk}

\usepackage{pgfplots}
\usepackage{pgf}
\usepackage{tikz}
\usetikzlibrary{patterns}
\usetikzlibrary{arrows.meta}
\usepgfplotslibrary{patchplots} 
\usetikzlibrary{pgfplots.patchplots} 
\pgfplotsset{width=9cm,compat=1.5.1}

\usepackage{amssymb, amsthm, amsmath, dsfont, mathtools}
\mathtoolsset{showonlyrefs}
\usepackage{latexsym}

\usepackage{eepic}
\usepackage{sectsty}
\usepackage{framed}
\usepackage[shortlabels]{enumitem}
\usepackage{float}
\restylefloat{table}
\usepackage{xifthen,mathrsfs}
\usepackage{ytableau}
\usepackage{subcaption}

\usepackage[pagecolor=white]{pagecolor}
\makeatletter
\newcommand{\globalcolor}[1]{%
  \color{#1}\global\let\default@color\current@color
}
\makeatother

\AtBeginDocument{\globalcolor{black}}

\usepackage{hyperref}
\usepackage{hyperref}
\hypersetup{
    colorlinks = true,
    linkcolor = magenta,
    citecolor = blue,
    urlcolor = blue
}

\hypersetup{pdfpagemode=UseNone}

\newtheorem{theorem}{Theorem}
\numberwithin{theorem}{section}
\newtheorem{lemma}[theorem]{Lemma}
\newtheorem{prop}[theorem]{Proposition}

\theoremstyle{definition}

\newtheorem{remark}[theorem]{Remark}

\newtheorem{example}[theorem]{Example}

\newcommand{\tinyspace}{\mspace{1mu}}

\newcommand{\tr}{\operatorname{Tr}}

\renewcommand{\t}{{\scriptscriptstyle\mathsf{T}}}

\newcommand{\abs}[1]{\lvert #1 \rvert}

\newcommand{\ip}[2]{\langle #1 , #2\rangle}

\newcommand{\bip}[2]{( #1 , #2)}

\newcommand{\floor}[1]{\lfloor #1 \rfloor}

\newcommand{\norm}[1]{\lVert\tinyspace #1 \tinyspace\rVert}

\newcommand{\I}{\mathds{1}}

\newcommand{\setft}[1]{\mathrm{#1}}

\newcommand{\Pos}{\setft{Pos}}

\newcommand{\Unitary}{\setft{U}}
\newcommand{\Orth}{\setft{O}}
\newcommand{\Herm}{\setft{Herm}}
\newcommand{\Sym}{\setft{Sym}}

\newcommand{\Trans}{\mathsf{T}}

\newcommand{\complex}{\mathbb{C}}
\newcommand{\field}{\mathbb{F}}
\newcommand{\real}{\mathbb{R}}
\renewcommand{\natural}{\mathbb{N}}
\newcommand{\integer}{\mathbb{Z}}

\newcommand\X{\mathcal{X}}

\newcommand\A{\mathcal{A}}
\newcommand\B{\mathcal{B}}
\newcommand\V{\mathcal{V}}
\newcommand\U{\mathcal{U}}
\newcommand\C{\mathcal{C}}

\newcommand{\End}{\setft{End}}
\newcommand{\Hom}{\setft{Hom}}

\DeclareMathOperator{\spn}{span}

\newcommand{\dg}{\dagger}

\newcommand{\ootimes}{ \otimes \cdots \otimes }

\newcommand{\bfd}{\mathbf{d}}

\newcommand{\bff}{\mathbf{f}}

\newcommand{\bfw}{\mathbf{w}}
\newcommand{\bfx}{\mathbf{x}}

\newcommand{\bfz}{\mathbf{z}}

\newcommand{\frakS}{\mathfrak{S}}

\newcommand{\bfalpha}{\boldsymbol{\alpha}}
\newcommand{\bfbeta}{\boldsymbol{\beta}}
\newcommand{\bfgamma}{\boldsymbol{\gamma}}

\newcommand{\injects}{\hookrightarrow}

\renewcommand{\over}[1]{\overline{#1}}

\def\ba#1\ea{\begin{align}#1\end{align}}

\newcommand{\bull}{\bullet}

\newcommand{\cl}{\setft{Cl}}

\newcommand{\bd}{{\bf d}}

\newcommand{\Vc}{\V_{\complex}}
\newcommand{\Vr}{\V_{\real}}

\newcommand{\pmin}{p_{\min}}
\newcommand{\Pmin}{P_{\min}}

\begin{document}

\emergencystretch 3em
\title{\bf A hierarchy of eigencomputations for polynomial optimization on the sphere}
%

 \author[$* \dagger$]{
   Benjamin Lovitz}
 \author[$\ddagger$]{Nathaniel Johnston \thanks{emails: benjamin.lovitz@gmail.com, njohnston@mta.ca}}
 \affil[$\dagger$]{Department of Mathematics, Northeastern University, Boston, Massachusetts, USA}
  \affil[$\ddagger$]{Department of Mathematics and Computer Science, Mount Allison University, Sackville,

  New Brunswick, Canada}

\date{\today}

\maketitle
\begin{abstract}
We introduce a convergent hierarchy of lower bounds on the minimum value of a real form over the unit sphere. The main practical advantage of our hierarchy over the real sum-of-squares (RSOS) hierarchy is that the lower bound at each level of our hierarchy is obtained by a minimum eigenvalue computation, as opposed to the full semidefinite program (SDP) required at each level of RSOS. In practice, this allows us to compute bounds on much larger forms than are computationally feasible for RSOS. Our hierarchy outperforms previous alternatives to RSOS, both asymptotically and in numerical experiments. We obtain our hierarchy by proving a reduction from real optimization on the sphere to Hermitian optimization on the sphere, and invoking the Hermitian sum-of-squares (HSOS) hierarchy. This opens the door to using other Hermitian optimization techniques for real optimization, and gives a path towards developing spectral hierarchies for more general constrained real optimization problems. To this end, we use our techniques to develop a hierarchy of eigencomputations for computing the real tensor spectral norm.


\end{abstract}
\newpage
\tableofcontents
\newpage
\section{Introduction}\label{intro}





Let $\Vr=\real^n$ and let $S^D(\Vr^*)$ be the set of homogeneous polynomials of degree $D$ on $\Vr$.\footnote{Technically, we define $S^D(\Vr^*)\subseteq (\Vr^*)^{\otimes D}$ to be the set of \textit{symmetric tensors}, where $\Vr^*$ denotes the dual vector space to $\Vr$ (the set of $n$-dimensional row vectors). A symmetric tensor $p \in S^D(\Vr^*)$ can be identified with a homogeneous polynomial $p(x):=(p,x^{\otimes D})$, where $x\in \Vr$ is a variable and $(\cdot, \cdot)$ denotes the bilinear pairing (dot product) between $S^D(\Vr^*)$ and $S^D(\Vr)$. This correspondence is introduced formally in Section~\ref{background}.\label{f1}} We consider the fundamental task of minimizing a polynomial ${p(x) \in S^D(\Vr^*)}$ over the unit sphere, i.e. computing
\begin{align}\label{eq:optimization_problem}
p_{\min}=\min_{\substack{x \in \Vr\\ \norm{x}=1}} p(x).
        \end{align}
Optimization problems of this form have applications in several areas~\cite{fang2021sum}. For example, for a special class of degree-three polynomials this corresponds to computing the largest stable set of a graph \cite{nesterov2003random,de2008complexity}. As another example, computing the $2 \rightarrow 4$ norm (i.e., hypercontractivity) of a matrix $A$ is equivalent to maximizing the degree-four polynomial $p(x)=\norm{Ax}_4^4$ on the unit sphere, and has many connections to problems in computational complexity, quantum information and compressive sensing ~\cite{barak2012hypercontractivity}. For $D=2$ this problem is equivalent to computing the minimum eigenvalue of a symmetric matrix, which can be solved efficiently. However, already for $D=3$ this problem is NP-hard as it contains the stable set problem as a special case~\cite{nesterov2003random}.

The \textit{real sum-of-squares} (RSOS) hierarchy is a hierarchy of semidefinite programs (SDPs) of increasing size whose optimum values approach $p_{\min}$ from below~\cite{reznick1995uniform,lasserre2001global}. We propose a hierarchy of minimum eigenvalue computations that also approaches $p_{\min}$ from below. At the $k$-th level of both our hierarchy and the RSOS hierarchy, the underlying space is the set of homogeneous polynomials of degree $2k$. However, our hierarchy has the advantage of merely requiring an eigencomputation at each level,
as opposed to the full SDP required at each level of RSOS. This allows us to go to much higher levels, and to handle much larger polynomials in practice. Our hierarchy is based on a reduction to {Hermitian} optimization on the sphere (a complex analogue to~\eqref{eq:optimization_problem}), along with the \textit{Hermitian sum-of-squares} (HSOS) hierarchy---a spectral hierarchy for Hermitian optimization on the sphere. By a similar reduction, our hierarchy extends to other optimization problems including computing the real tensor spectral norm.


We prove that our hierarchy converges additively as $O(1/k)$ in the level $k$. This is quadratically slower than the best-known convergence $O(1/k^2)$ of RSOS when $D\leq 2n$~\cite{fang2021sum}. When $D > 2n$ the result of~\cite{fang2021sum} does not apply, and the best-known convergence of RSOS is $O(1/k)$~\cite{reznick1995uniform, doherty2012convergence}.\footnote{See also \cite{BhattiproluGGLT17} which gives multiplicative approximation guarantees for polynomial optimization over the sphere using the sum of squares hierarchy.} Several other ``cheaper" alternatives to RSOS are known for the optimization problem~\eqref{eq:optimization_problem}. One such alternative is the \textit{diagonally-dominant sum-of-squares} (DSOS) hierarchy of~\cite{AM19}, which relies only on linear programming at each level, but does not converge in general~\cite[Proposition 3.15]{AM19}. More relevant to our work is the \textit{harmonic hierarchy} of~\cite{cristancho2024harmonic}. This hierarchy is also ``optimization free," and is proven to converge as $O(1/k^2)$ in the level $k$. However, the $k$-th level requires discrete minimization of a certain degree-$2k$ polynomial over a cubature rule of size exponential in $n$, which is prohibitively expensive when $n$ is large. Furthermore, it relies on an arbitrary choice of cubature rule and kernel, as opposed to our approach which is completely canonical.


In Section~\ref{sec:numerics} we describe an implementation of our hierarchy in MATLAB, and present several examples to demonstrate its performance in comparison to the RSOS, DSOS, and harmonic hierarchies. In every example, our hierarchy obtains better bounds than the harmonic hierarchy per unit of computational time. If the polynomial is small enough that the RSOS hierarchy can be run, RSOS outperforms all other hierarchies. For larger problems (e.g., degree $4$ polynomials in more than $25$ variables), our hierarchy outperforms the rest.

More generally, we consider minimizing a real multi-homogeneous form over the product of unit spheres, or (by a similar identification as described in Footnote~\ref{f1}) minimizing the inner product between a real tensor and a unit partially-symmetric product tensor (i.e. an element of the \textit{spherical Segre-Veronese variety}). We develop an analogous hierarchy of eigencomputations for this setting, which also converges as $O(1/k)$. In particular, this gives hierarchies for computing the real spectral norm of a real tensor, and for minimizing a biquadratic form over the unit sphere. Computing the tensor spectral norm is a well-studied problem with connections to planted clique, tensor PCA and tensor decompositions ~\cite{FriezeK08, BrubakerV09, RichardM14}. These results are presented in Section~\ref{sec:tensor}.

In Section~\ref{constraints} we consider minimizing a real form under more general constraints. Spectral hierarchies for more general constrained Hermitian optimization problems are known~\cite{catlin1999isometric,DJLV24}, and our techniques suggest a plausible route towards converting these into spectral hierarchies for their real analogues.


In the remainder of this introduction we review the RSOS and HSOS hierarchies for the problem~\eqref{eq:optimization_problem} and its Hermitian analogue, respectively. We then describe our hierarchy, which we call the HRSOS hierarchy. For the sake of readability, we will defer some definitions until Section~\ref{background}.

\subsection{The real sum-of-squares hierarchy (RSOS)}
In this section we review the RSOS hierarchy for the optimization problem~\eqref{eq:optimization_problem}. By~\cite[Lemma B.2]{doherty2012convergence}, we can (and will) assume $D$ is even, and let $d=D/2$.
The RSOS hierarchy is a hierarchy of lower bounds on $p_{\min}$ which can be computed by semidefinite programming. Let $\Sigma_{n,k}\subseteq S^{2k}(\Vr^*)$ (or simply $\Sigma_k$ when $n$ is understood) be the set of forms which are sums of squares of degree-$k$ forms.



\begin{theorem}[RSOS hierarchy]\label{thm:intro_rsos}
Let $p \in S^{2d}(\Vr^*)$ be a real form of degree $2d$. For each $k \geq d$, let $\gamma_k$ be the optimum value of the SDP
\begin{align}\begin{split}\label{eq:SOS1}
        \textup{maximize:} \quad& \ \gamma \\
        \textup{subject to:}\quad& \ p(x) \cdot \norm{x}^{2(k-d)}-\gamma \cdot \norm{x}^{2k} \in \Sigma_{k}.
        \end{split}
        \end{align}
Then $\gamma_1 \leq \gamma_2 \leq \dots $ and $\lim_k \gamma_k = p_{\min}$. Furthermore, $p_{\min}-\gamma_k= O(1/k)$ when $d > n$, and $p_{\min}-\gamma_k= O(1/k^2)$ when $d\leq n$ (suppressing all but the $k$-dependence).
\end{theorem}

The limiting statement was proven in~\cite{reznick1995uniform} (see also~\cite{lasserre2001global}). The convergence rates were proven in~\cite{reznick1995uniform, doherty2012convergence,fang2021sum}. To see equivalences between this formulation and other, perhaps more familiar, forms of the RSOS hierarchy, see e.g.~\cite[Proposition 2]{de2005equivalence} or~\cite[Lemma 1.3]{laurent2019notes}.

\subsection{The Hermitian sum-of-squares hierarchy (HSOS)}
Let $\V_{\complex}=\complex^n$. A \textit{Hermitian form} of bidegree $(d,d)$ is an element $H(z,w^\dg) \in S^d(\Vc^*)\otimes S^d(\Vc)$ for which $H(z,w^\dg)=\over{H(w,z^\dg)}$. Let $\Herm(S^d(\Vc))$ be the set of Hermitian forms of bidegree $(d,d)$. By standard polarization arguments, $\Herm(S^d(\Vc))$ can be identified with the set of Hermitian \textit{operators} on the vector space $S^d(\Vc)$ (see e.g.~\cite{d2011hermitian,d2019hermitian} or Section~\ref{sec:herm}). A Hermitian analogue to~\eqref{eq:optimization_problem} is given by
\ba\label{eq:hermitian_optimization_problem}
H_{\min}=\min_{\substack{z \in \Vc\\ \norm{z}=1}} H(z,z^\dg).
\ea
The \textit{Hermitian square} of a form $q(z) \in S^k(\Vc^*)$ is given by $|q(z)|^2 \in \Herm(S^k(\Vc))$. Let $\Pos(S^k(\Vc)) \in \Herm(S^k(\Vc))$ be the set of sums of Hermitian squares of degree-$k$ forms. This can be identified with the set of positive semidefinite \textit{operators} on the vector space $S^k(\Vc)$.

\begin{theorem}[HSOS hierarchy]\label{thm:intro_hsos}
Let $H \in \Herm(S^d(\Vc))$ be a Hermitian form. For each $k\geq d$, let $\mu_k$ be the optimum value of the SDP
\begin{align}\begin{split}\label{eq:HSOS1}
        \textup{maximize:} \quad& \ \mu \\
        \textup{subject to:}\quad& \ H(z,z^\dg) \cdot \norm{z}^{2(k-d)}-\mu \cdot \norm{z}^{2k} \in \Pos(S^k(\Vc)).
        \end{split}
        \end{align}
Then $\mu_d \leq \mu_{d+1} \leq \dots $ and $\lim_k \mu_k = H_{\min}$. Furthermore, $H_{\min} - \mu_k = O(1/k)$ (suppressing all but the $k$-dependence).
\end{theorem}
The limiting statement was proven in~\cite{quillen1968representation,catlin1995stabilization}, and the $O(1/k)$ convergence rate was proven in~\cite{to2006effective}. This result can also be proven using so-called \textit{quantum de Finetti theorems}~\cite{hudson1976locally,christandl2007one}, and we will do so in this work.

\begin{remark}Let $H_k(z,z^\dg):=H(z,z^\dg) \cdot \norm{z}^{2(k-d)} \in \Herm(S^k(\Vc))$. The optimum value $\mu_k$ of the SDP~\eqref{eq:HSOS1} is precisely $\lambda_{\min}(H_k)$, where $\lambda_{\min}(\cdot)$ denotes the minimum eigenvalue of the input viewed as a Hermitian operator. So the HSOS hierarchy is in fact a hierarchy of eigencomputations for optimizing a Hermitian form over the sphere.\end{remark}

While RSOS is a hierarchy of full semidefinite programs, HSOS is a hierarchy of mere eigencomputations. This leads us to ask, can we use the HSOS hierarchy for real optimization? The answer is yes.

\subsection{Our polynomial optimization hierarchy (HRSOS)}
We prove a reduction from the real optimization problem~\eqref{eq:optimization_problem} to the Hermitian optimization problem~\eqref{eq:hermitian_optimization_problem}, and then use the HSOS hierarchy to obtain a hierarchy of eigencomputations for~\eqref{eq:optimization_problem}.

Let $\V_{\complex}=\complex^n$ (the complexification of $\Vr$). For a form $p \in S^{2d}(\Vr^*)$, let $P \in \End(S^d(\Vr))$ be the \textit{maximally symmetric Gram operator} of $p$ (see Section~\ref{background},~\cite[Section 3]{doherty2012convergence}, or~\cite[Lemma 1.2]{laurent2019notes}). Since $P$ is real-symmetric, we can view it as a Hermitian operator. Recall the definitions of $p_{\min}$ and $P_{\min}$ from~\eqref{eq:optimization_problem} and~\eqref{eq:hermitian_optimization_problem}, respectively.



\begin{theorem}\label{thm:intro_reduction}
For any form $p \in S^{2d}(\Vr^*)$, it holds that
\ba
P_{\min} \leq p_{\min} \leq \frac{P_{\min}}{\delta(d)},
\ea
where $\delta(d):={2^d}{\binom{2d}{d}^{-1}}$.
\end{theorem}

In particular, this theorem shows that ($\pmin > 0 \iff \Pmin > 0$). By Theorem~\ref{thm:intro_hsos}, we immediately obtain the following hierarchy of eigencomputations to determine if $\pmin>0$. For a positive integer $k \geq d$, let $P_k(z,z^\dg) := P(z,z^\dg) \cdot \norm{z}^{2(k-d)}\in \Herm(S^{k}(\Vc))$.

\begin{theorem}\label{thm:intro_simple}
Let $p \in S^{2d}(\V^*)$. If $\pmin >0$, then $P_k \in \Pos(S^k(\Vc))$ for $k$ large enough.
\end{theorem}

\begin{remark}
Viewed as as real operator, $P_k$ is a \textit{Gram operator} for $p(x) \norm{x}^{2(k-d)}$, meaning that
\ba
\ip{x^{\otimes k}}{P_k x^{\otimes k}}=p(x) \norm{x}^{2(k-d)}
\ea
for all $x \in \Vr$. The RSOS hierarchy says that if $\pmin>0$, then there exists a positive semidefinite Gram operator for $p(x) \norm{x}^{2(k-d)}$ for $k$ large enough. By contrast, Theorem~\ref{thm:intro_simple} says that the \textit{single} Gram operator $P_k$ is positive semidefinite for $k$ large enough. We note that Theorem~\ref{thm:intro_simple} does not hold for other choices of Gram operators for $p$; see Remark~\ref{rmk:nec}.
\end{remark}

More generally, we obtain a hierarchy of eigencomputations for $\pmin$ with $O(1/k)$ convergence.
Let $N^{(d)}$ be the maximally symmetric Gram operator of $\norm{x}^{2d}$, viewed as a Hermitian form, and let
\ba
N_k^{(d)}(z,z^\dg):=N^{(d)}(z,z^\dg)\cdot  \norm{z}^{2(k-d)} \in \Herm(S^k(\Vc))
\ea
when $k \geq d$.  We obtain the following hierarchy of eigencomputations for $\pmin$:
\begin{theorem}[HRSOS hierarchy]\label{thm:intro_hrsos}
Let $p \in S^{2d}(\Vr^*)$. For each $k \geq d$, let
\ba
\eta_k=\lambda_{\min}((N_k^{(d)})^{-1/2}P_k (N_k^{(d)})^{-1/2}).
\ea
Then $\eta_d \leq \eta_{d+1} \leq \dots$ and $\lim_k \eta_k = \pmin$. Furthermore, $\pmin-\eta_k = O(1/k)$ (suppressing dependence on $n,d,$ and $\norm{P}_{\infty}$).
\end{theorem}

We prove this using Theorems~\ref{thm:intro_hsos} and~\ref{thm:intro_reduction}. In the theorem statement, $\norm{P}_{\infty}$ is the operator norm of $P$. We note that $\eta_k$ is equal to the minimum generalized eigenvalue of $P_k$ and $N_k^{(d)}$ (see Section~\ref{sec:numerics}). This avoids the potentially costly computation of $(N_k^{(d)})^{-1/2}$.

\subsection{Our tensor optimization hierarchy (m-HRSOS)}

More generally, our techniques can be used to minimize multi-homogeneous forms over a product of spheres. Let $n_1,\dots, n_m$ and $D_1,\dots, D_m$ be positive integers, let $\V_i = \real^{n_i}$ for each $i \in [m]$, and let $p\in S^{D_1}(\V_1^*) \ootimes S^{D_m}(\V_m^*)$ be a multi-homogeneous form on $\V_1,\dots, \V_m$ of multidegree $D_1,\dots, D_m$.\footnote{Similarly as in Footnote~\ref{f1}, we technically define $S^{D_1}(\V_1^*) \ootimes S^{D_m}(\V_m^*)$ to be the tensor product of symmetric spaces, which we formally identify with the set of multi-homogeneous polynomials in Section~\ref{background}.} Consider the minimization problem
\ba\label{eq:pminbdintro}
p_{\min}:= \min_{\substack{v_j\in \V_j\\\norm{v_j}=1}} p(v_1,\dots, v_m).
\ea
As mentioned above, this corresponds to minimizing a real tensor over the real spherical Segre-Veronese variety. Our hierarchy extends naturally to this setting, with similar convergence guarantees. We can (and will) assume that each $D_j=2d_j$ is even (see Proposition~\ref{prop:even_gen}). Similarly to before, we associate to $p$ a multi-homogeneous Hermitian form $P \in \Herm(S^{d_1}(\V_1^\complex)\ootimes S^{d_m}(\V_m^{\complex}))$, where $\V_i^{\complex}:=\complex^{n_i}$ (the complexification of $\V_i$). Letting
\ba\label{eq:multiherm}
P_{\min}:=\min_{\substack{\bfz = (z_1,\dots, z_m)\\ z_j \in \V_j^{\complex}\\ \norm{z_j}=1}} P(\bfz,\bfz^{\dg}) 
\ea
be the  Hermitian analogue of~\eqref{eq:pminbdintro}, we prove the following reduction from the real optimization problem~\eqref{eq:pminbdintro} to the Hermitian optimization problem~\eqref{eq:multiherm}:

\begin{theorem}
For any $p \in S^{2d_1}(\V_1^*)\ootimes S^{2d_m}(\V_m^*)$, it holds that
\ba\label{eq:real_optbd_intro}
P_{\min} \leq p_{\min} \leq \frac{P_{\min}}{\delta(\bd)},
\ea
where $\delta(\bd)=\delta(d_1)\cdots \delta(d_m)$, and $\delta(d_i)$ is defined as in Theorem~\ref{thm:intro_reduction}.
\end{theorem}

The quantum de Finetti theorem implies that the optimization problem~\eqref{eq:multiherm} can be solved by a hierarchy of eigencomputations converging as $O(1/k)$ in the level $k$. We combine this with Theorem~\ref{eq:real_optbd_intro} to prove analogous statements to Theorems~\ref{thm:intro_simple} and~\ref{thm:intro_hrsos}; see Section~\ref{sec:tensor}.

\subsection{Related work}\label{sec:related_work}

Due to the limited capacity of current SDP solvers, there has been much work on developing practical alternatives to the sum-of-squares hierarchy for constrained polynomial optimization problems~\cite{helmberg2000spectral,chandrasekaran2016relative,lasserre2017bounded,dressler2017positivstellensatz,AM19,wang2020second,mai2022exploiting,mai2023hierarchy,cristancho2024harmonic}. In particular, several techniques have been developed for optimization over the real sphere~\cite{doherty2012convergence,laurent2019notes,fang2021sum,cristancho2024harmonic}.

Notably, an optimization-free \textit{harmonic hierarchy} is developed in~\cite{cristancho2024harmonic}. This is closely related to harmonic analysis on spheres, and builds on previous work of~\cite{blekherman2004convexity,fang2021sum,slot2022sum,slot2023sum}. This hierarchy depends on a choice of cubature rule and kernel, and the authors give an explicit choice for which the harmonic hierarchy converges (multiplicatively) as $O(1/k^2)$ in the level $k$, with the $k$-th level minimizing a certain degree-$2k$ polynomial (obtained from $p$ and a choice of kernel) over a cubature rule of size $2(k+1)^{n-1}$. While their convergence in the level is faster than ours, the discrete optimization required at each level is prohibitively expensive when $n$ is large. It is possible that a smaller cubature rule could lead to computational savings, although the problem of explicitly constructing small cubature rules is largely unresolved (see~\cite{cristancho2024harmonic} and the references therein). Numerical experiments in Section~\ref{sec:numerics} show that even for small examples such as the Motzkin polynomial, the harmonic hierarchy is significantly slower than ours in terms of time budget needed to achieve a given bound. We note that our hierarchy has the additional advantages of being completely canonical, and generalizing easily to optimization over the Segre-Veronese variety.

Our hierarchy uses Hermitian optimization to solve real optimization problems. In the reverse direction, one can view a Hermitian form in $\Herm(S^d(\complex^n))$ as a real form in $S^{2d}(\real^{2n})$ by the map $H(z,z^\dg)\mapsto p(x,y)=H(x+iy,x^\t-iy^\t)$. (The natural extension of this map to the space of Hermitian polynomials is invertible, but does not preserve homogeneity in the reverse direction.) In particular, the real sum-of-squares hierarchy can be used for Hermitian optimization. This observation has been used in several works, e.g.~\cite{d2009polynomial,harrow2017improved,fang2021sum,wang2023real}.

Refs.~\cite{d2011hermitian, d2019hermitian} are excellent sources of background material on Hermitian optimization. The HSOS hierarchy for optimizing a Hermitian form over the sphere was first developed in~\cite{quillen1968representation} and then independently in~\cite{catlin1995stabilization}. This hierarchy was proven to converge additively as $O(1/k)$ in~\cite{to2006effective}. By the duality between separable quantum states and non-negative Hermitian forms elucidated in e.g.~\cite{fang2021sum,muller2023refinement}, some results known as~\textit{quantum de Finetti theorems} can be used for Hermitian optimization. In particular, the result of~\cite{hudson1976locally} can be used to obtain the HSOS hierarchy, and the result of~\cite{christandl2007one} can be used to prove an effective version similar to that of~\cite{to2006effective}. See Theorem~\ref{thm:hsos}. See~\cite{drouot2013quantitative} and the references therein for further effective versions of Quillen's result.


Hierarchies for \textit{constrained} Hermitian optimization problems are developed in~\cite{catlin1999isometric,d2009polynomial,DJLV24} and an effective version of~\cite{catlin1999isometric} is obtained in~\cite{tan2021effective}, although the quantities involved in their bounds are difficult to calculate. 


Various techniques have been developed to approximate the tensor spectral norm. The tensor spectral norm was shown to be NP-Hard in~\cite{HL}. Sum-of-squares hierarchies based on semidefinite programming are proposed and analyzed in~\cite{nie2014semidefinite,harrow2017improved,fang2021sum}. Locally convergent techniques have also been developed, including alternating least squares~\cite{de2000best,FMPS13} and the higher order power method~\cite{de1995higher}.


The works~\cite{HopkinsSSS16,schramm2017fast} give spectral alternatives to the real sum-of-squares hierarchy for tensor decompositions and finding planted sparse vectors, although these are analyzed in the average-case setting. Further alternatives to the real sum-of-squares hierarchy for constrained polynomial optimization were developed in~\cite{helmberg2000spectral,mai2022exploiting},
although these still require minimizing the largest eigenvalue of a matrix pencil. A spectral hierarchy of \textit{upper bounds} on constrained polynomial minimization problems was developed in~\cite{lasserre2011new}, and its convergence was analyzed in~\cite{de2017convergence,de2022convergence}. For optimization on the sphere, a convergence rate of $\Theta(1/k^2)$ was proven in~\cite{de2022convergence}.

%
%
%
%
%
%
%
%
%

\section{Background and notation}\label{background}


In this section we describe some necessary background and notation for the remainder of the paper. We prefer to work coordinate-independently, in order to ensure that our approach does not depend on an arbitrary choice of basis. In an effort to appeal to readers who are unconcerned with such technicalities, we also describe how these definitions manifest in coordinates.

\subsubsection*{Finite-dimensional vector spaces, traces, and partial traces}

For vector spaces $\U=\field^{n_1}, \V=\field^{n_2}$ over a field $\field$, let $\Hom(\U,\V)$ be the set of $n_2 \times n_1$ matrices, and let $\End(\V)=\Hom(\V,\V)$ be the set of $n_1 \times n_1$ matrices. More generally, for finite-dimensional $\field$-vector spaces $\U, \V$ let $\Hom(\U,\V)$ be the set of linear maps (homomorphisms) from $\U$ to $\V$, and let $\End(\V)=\Hom(\V,\V)$ be the set of endomorphisms of $\V$.

Let $\I_{\V}\in \End(\V)$ be the identity map on $\V$.

For $\V=\field^n$, let $\V^*$ be the set of row vectors of length $n$, and let $(\cdot, \cdot) : \V^* \times \V \rightarrow \field$ be the dot product between row and column vectors. More generally, for a finite-dimensional $\field$-vector space $\V$ let $\V^*:=\Hom(\V,\field)$ be the \textit{dual vector space} to $\V$, and let $(\cdot, \cdot) : \V^* \times \V \rightarrow \field$ be the bilinear form $(f,v) = f(v)$. 



The \textit{trace map} $\tr \in \End(\V)^*$ is the map that sends an operator $M$ to its trace. For an operator ${M \in \End(\V^{\otimes d})},$ let
\ba
\tr_1(M):=(\tr \otimes \I_{\End(\V^{\otimes d-1})})(M) \in \End(\V^{\otimes {d-1}})
\ea
be the partial trace of $M$ in the first factor. Similarly, let $\tr_j(M)$ be the partial trace of $M$ in the $j$-th factor. For a subset $S \subseteq [d]$, let $\tr_S(M) \in \End(\V^{\otimes {d-|S|}})$ be the partial trace of $M$ in the factors indexed by $S$. For an integer $j \in [d]$, let $[j\cdot \cdot d]=\{j,j+1,\dots, d\} \subseteq [d]$. See also~\cite{doherty2012convergence} or~\cite{Wat18}.

\subsubsection*{Finite-dimensional Hilbert spaces, Euclidean adjoint, and Hermitian adjoint}

For a real matrix $A$ we let $A^\t$ be the transpose of $A$, and for a real vector $x$ we let $x^\t$ be the transpose of $x$. For a complex matrix $M$ we let $M^{\dg}$ be the conjugate-transpose of $M$, and for a complex vector $z$ we let $z^{\dg}$ be the conjugate-transpose of $z$.

More generally, let $\field\in \{\real,\complex\}$ and let $\V_{\field}, \U_{\field}$ be finite dimensional Hilbert spaces over $\field$ with positive definite inner products $\ip{\cdot}{\cdot}$ (which we take to be sesquilinear in the first argument if $\field=\complex$). For $A \in \Hom(\Vr,\U_{\real})$, let  $A^{\t} \in \Hom(\U_{\real},\Vr)$ be the Euclidean adjoint defined by $\ip{A^\t u}{v}=\ip{u}{Av}$, and for $M \in \Hom(\Vc,\U_{\complex})$ let $M^{\dg} \in \Hom(\U_{\complex},\Vc)$ be the Hermitian adjoint defined by $\ip{M^{\dg} u}{v}=\ip{u}{M v}$. In particular, for a vector $x\in \Vr$ we have $x^{\t}=\ip{x}{-}\in\Vr^*$, and for a vector $z \in \Vc$ we have $z^{\dg}=\ip{z}{-}\in \Vc^*$.


\subsubsection*{Complexification}

Given a real vector space, \textit{complexification} is a canonical way to view it as a subset of a complex vector space. For example, the complexification of $\real^n$ is $\complex^n$. The notion of complexification of an abstract vector space is also important, as it helps ensure that our approach is completely canonical (i.e. is not basis-dependent). The reader unconcerned with such technicalities can safely skip the rest of this subsection.

The \textit{complexification} of a finite dimensional real Hilbert space $\Vr$ is the complex Hilbert space $\Vr \otimes_{\real} \complex$ with inner product extended from $\Vr$ sesquilinearly.

If $\Vc$ is the complexification of $\Vr$, and $M=A+iB \in \End(\Vc)$ with $A,B \in \End(\Vr)$, then $M^{\t} \in \End(\Vc)$ is well-defined by $M^{\t}=A^{\t}+iB^{\t}$. If $\Vr=\real^n$, then $M^{\t}$ is the transpose of $M$.

\subsubsection*{Symmetric, Hermitian, and positive semidefinite operators}

For finite dimensional Hilbert spaces $\Vr=\real^n$ and $ \Vc=\complex^n$, let $\Sym(\Vr) \subseteq \End(\Vr)$ be the set of \textit{symmetric operators} satisfying $M^{\t}=M$, and let $\Herm(\Vc) \subseteq \End(\Vc)$ be the set of \textit{Hermitian operators} satisfying $M^{\dg}=M$. Let $\setft{O}(\Vr) \subseteq \End(\Vr)$ be the set of \textit{orthogonal operators} satisfying $M M^{\t}=\I_{\Vr}$.

Let $\V=\field^n$ with $\field \in \{\real, \complex\}$. We says that an operator $M\in \End(\V)$ is \textit{positive semidefinite} $M \succeq 0$ if $\ip{v}{Mv} \geq 0$ for all $v \in \V$. Let $\Pos(\V)$ be the set of positive semidefinite operators on $\V$. We say that $M$ is \textit{positive definite} $M \succ 0$ if $\ip{v}{Mv} > 0$ for all $v \neq 0$.

\subsubsection*{Viewing a real operator as a complex operator}

An $n \times n$ real matrix can be viewed as an $n \times n$ complex matrix with real entries. In other words, a real matrix in $M\in \End(\real^n)$ can be viewed as a matrix $M \in \End(\complex^n)$, and we often abuse notation in this way. In this subsection, we make this identification precise for abstract real Hilbert spaces.

Let $\Vr$ be a real Hilbert space with complexification $\Vc$. To a real operator $M \in \End(\Vr)$ we associate the complex operator $M \otimes \I_{\complex} \in \End(\Vc)$, and frequently abuse notation by denoting this operator also as $M \in \End(\Vc)$. 



\subsubsection*{The operator norm and trace norm}

Let $\norm{M}_{\infty}= \max_{\norm{v}=1} \ip{v}{Mv}$ be the \textit{operator norm} of $M$ and let $\norm{M}_1=\max_{\norm{N}_\infty=1} \tr(N^{\Delta} M)$ be the \textit{trace norm} of $M$, where $\Delta=\t$ if $\field=\real$ and $\Delta=\dg$ if $\field=\complex$.


\subsubsection*{Homogeneous polynomials as symmetric tensors}

Let $\V=\field^n$ where $\field \in \{\real, \complex\}$. The permutation group $\frakS_d$ acts on $\V^{\otimes d}$ by permuting tensor factors
\ba
\sigma \cdot (v_1 \otimes \dots \otimes v_{d})=v_{\sigma^{-1}(1)}\otimes \dots \otimes v_{\sigma^{-1}(d)},
\ea
extended linearly. Define the \textit{symmetric subspace} $S^{d}(\V) \subseteq \V^{\otimes d}$ to be the linear subspace of tensors invariant under $\frakS_d$. Let $\Pi_d=\Pi_{\V,d} = \frac{1}{d!} \sum_{\sigma \in \frakS_d} \sigma $ be the orthogonal projection onto the symmetric subspace.
%
%

As mentioned in the introduction, we can identify $S^{d}(\V^*)\subseteq (\V^*)^{\otimes d}$ with the space of homogeneous polynomials of degree $d$ on $\V$ via the isomorphism that associates a symmetric tensor $p$ with the homogeneous polynomial $p(v):=(p, v^{\otimes d})$. In the reverse direction, this isomorphism sends a monomial $x_{i_1}\cdots x_{i_d}$ to the symmetric tensor $\frac{1}{d!} \sum_{\sigma \in \frakS_d} x_{i_{\sigma(1)}} \ootimes x_{i_{\sigma(d)}}$. In an abuse of notation, we will invoke this isomorphism without comment, and refer to both the symmetric tensor and homogeneous polynomial by the same symbol $p$. We will also refer to elements of $S^d(\V^*)$ as both symmetric tensors and homogeneous polynomials (or forms).

For $q \in S^c(\V^*)$ we let $p \cdot q := \Pi_{c+d}(p\otimes q) \in S^{c+d}(\V^*)$, which is the symmetric tensor that corresponds to the product of the polynomials $p$ and $q$. This is easily verified by noting that $(\Pi_{c+d}(p\otimes q), v^{\otimes {c+d}})=p(v) q(v)$.

\subsubsection*{Multi-homogeneous polynomials as tensors}

More generally, we also view the tensor product space $S^{d_1}(\V_1^*)\ootimes S^{d_m}(\V_m^*)$ as the space of multi-homogeneous polynomials on $\V_1,\dots, \V_m$ of multi-degree $d_1,\dots, d_m$ via the isomorphism that associates a tensor  $p \in S^{d_1}(\V_1^*)\ootimes S^{d_m}(\V_m^*)$ with the multi-homogeneous polynomial
\ba
p(v_1,\dots, v_m):=\bip{p}{v_1^{\otimes d_1}\ootimes v_m^{\otimes d_m}}.
\ea
We will invoke this isomorphism without comment, and refer to both the tensor and multi-homogeneous form by the same symbol $p$.

For $q \in S^{c_1}(\V_1^*)\ootimes S^{c_m}(\V_m^*)$, we let
\ba
p \cdot q := (\Pi_{c_1+d_1} \ootimes \Pi_{c_m+d_m})(p \otimes q),
\ea
which corresponds to the product of multi-homogeneous polynomials.


\subsection{The maximally symmetric Gram operator $M(p)$}

Let $\Vr=\real^n$. For a real form $p \in S^{2d}(\Vr^*)$ of even degree, we say that an operator $M \in \Sym(S^d(\Vr))$ is a \textit{Gram operator} for $p$ if $\ip{x^{\otimes d}}{M(p)x^{\otimes d}}=p(x)$ for all $x \in \Vr$. We note that this differs from the usual definition, which requires a Gram operator to be positive semidefinite.

We define the \textit{maximally symmetric Gram operator} $M(p) \in \Sym(S^d(\Vr))$ to be the operator obtained from $p$ by the sequence of maps
\ba\label{eq:injects}
S^{2d}(\Vr^*) \injects S^d(\Vr^*) \otimes S^d(\Vr^*) \cong S^d(\Vr^*) \otimes S^d(\Vr) \cong \End(S^d(\Vr)).
\ea
Here, the inclusion is obvious, the first isomorphism uses the identification $\Vr \cong \Vr^*$ via the inner product, and the second isomorphism is canonical. $M(p)$ is precisely the operator introduced in~\cite[Section 3]{doherty2012convergence} and \cite[Lemma 1.2]{laurent2019notes}. $M(p)$ is also called the \textit{central catalecticant} of $p$. We will frequently let $\Vc=\complex^n$ (the complexification of $\Vr$), and regard $M(p)$ as an element of $\Herm(S^d(\Vc))$.

In coordinates, $M(p)$ is a (generalized) Hankel matrix. See Appendix~\ref{ap:mk}.

\subsubsection*{Partial transpose characterization of $M(p)$}
For $M \in \End(\Vr^{\otimes d})$, let
\ba
M^{\t_1}:=(\Trans \otimes \I_{\End(\Vr^{\otimes d-1})})(M).
\ea
be the partial transpose (or \textit{partial Euclidean adjoint}) of $M$. Note that if $M\in \End(S^{d}(\Vr))$ and $M^{\t_1}=M$, then $M$ is \textit{maximally symmetric} in the sense of~\cite[Section 3]{doherty2012convergence} or~\cite[Lemma 1.2]{laurent2019notes}, because the group $\frakS_d \times \frakS_d$ combined with any interleaving transposition generates $\frakS_{2d}$. The following proposition is straightforward.

\begin{prop}\label{prop:sym}
For a real form $p \in S^{2d}(\Vr^*)$ of even degree, $M(p) \in \Sym(S^{d}(\Vr))$ is the unique Gram operator for which $M(p)^{\t_1}=M(p)$.
\end{prop}


\subsubsection*{$M(\norm{x}^{2d})$ is positive definite}
We will also require the following proposition.

\begin{prop}\label{prop:Mspos}
It holds that $M(\norm{x}^{2d})$ is positive definite.
\end{prop}
\begin{proof}
We need to prove that $\ip{\psi}{M(\norm{x}^{2d}) \psi} >0$ for all nonzero $\psi \in S^d(\Vr)$. This follows from
\ba
\ip{\psi}{M(\norm{x}^{2d}) \psi}& = \bip{(\psi^{\t})^{2}}{\norm{x}^{2d}}\\
&= \int_{\substack{x\in \Vr \\ \norm{x}=1}} (\psi^{\t})^2(x) d\mu(x)\\
&>0,
\ea
where $\mu$ is a normalized Lebesgue measure on the sphere. 
The second line is~\cite[Proposition 6.6]{reznick1995uniform}. This completes the proof.
\end{proof}

\subsubsection*{Generalizing the maximally symmetric Gram operator to tensors}

More generally, for $\V_i=\real^{n_i}, i=1\dots,m$,
and
\ba
p\in S^{2d_1}(\V_1^*)\ootimes S^{2d_m}(\V_m^*),
\ea
we let
\ba
M(p)\in \End(S^{d_1}(\V_1)\ootimes S^{d_m}(\V_m))
\ea
be the operator obtained from $p$ by the inclusion
\ba
S^{2d_1}(\V_1^*)\ootimes S^{2d_m}(\V_m^*) \injects \End(S^{d_1}(\V_1)\ootimes S^{d_m}(\V_m)),
\ea
similarly to~\eqref{eq:injects}.

\section{Hermitian optimization}\label{sec:herm}

In this section we review some background material on Hermitian forms, describe the maximally symmetric Gram operator as a Hermitian form, and state the HSOS hierarchy in more details. We then extend all of this to \textit{multi-homogeneous Hermitian optimization}, which we will use in Section~\ref{sec:tensor} to obtain a spectral hierarchy for the tensor spectral norm.

\subsection{Hermitian forms}
Let $\Vc=\complex^n$. For an endomorphism $R \in \End(S^d(\Vc))$ and vectors $z,w \in \Vc$, let $R(z,w^\dg):= \ip{w^{\otimes d}}{R z^{\otimes d}}$. The following proposition is standard, and shows that $R$ is uniquely determined by $R(z,z^\dg)$. See also~\cite{d2011hermitian,d2019hermitian} and~\cite[Proposition 7.1]{DJLV24}.
\begin{prop}\label{prop:polarization}
Let $R \in \End(S^d(\Vc))$. If $R(z,z^\dg)=0$ for all $z \in \Vc$, then $R=0$.
\end{prop}
A \textit{Hermitian form} of bidegree $(d,d)$ is an element $H \in \End(S^d(\Vc))$ for which 
\ba
H(z,w^\dg)=\over{H(w,z^\dg)}
\ea
for all $z,w\in \Vc$. Let $\Herm(S^d(\Vc))\subseteq \End(S^d(\Vc))$ be the set of Hermitian forms of bidegree $(d,d)$. An immediate consequence of Proposition~\ref{prop:polarization} is that $H(z,z^\dg)=\sum_{|\bfalpha|=|\bfbeta|=d} C_{\bfalpha,\bfbeta} z^{\bfalpha} z^{\dg \bfbeta}$ is Hermitian if and only if it is Hermitian as an endomorphism, i.e. the matrix $(C_{\bfalpha,\bfbeta})_{(\bfalpha,\bfbeta)}$ is Hermitian (see also~\cite[Proposition 1.1]{d2011hermitian}). Let $\Pos(S^{d}(\Vc))$ be the set of \textit{Hermitian sums-of-squares}: Hermitian forms of the form $P(z,z^\dg)=\sum_{i} |q_i(z)|^2$ for some $q_i \in S^{d}(\Vc^*)$. These are precisely the forms that are positive semidefinite as endomorphisms, i.e. the matrix $(C_{\bfalpha,\bfbeta})_{(\bfalpha,\bfbeta)}$ is positive semidefinite.

Let $H \in \Herm(S^d(\Vc))$ be a Hermitian form, and let
\ba
H_k(z,z^\dg)=H(z,z^\dg) \norm{z}^{2(k-d)}
\ea
when $k \geq d$. This is a Hermitian form with operator $H_k=\Pi_k(H\otimes \I_{\Vc}^{\otimes k-d})\Pi_k$, where $\Pi_k$ is the orthogonal projection onto $S^k(\Vc)$. Indeed, $\ip{z^{\otimes k}}{H_k z^{\otimes k}}=H(z,z^\dg)\norm{z}^{2(k-d)}$ under this choice.

%

\subsubsection*{Hermitian forms associated to a real form}
Let $\Vr=\real^n$ and $\Vc=\complex^n$ (the complexification of $\Vr$), let $p \in S^{2d}(\Vr^*)$ be a real form, and let $P=M(p)$ be the maximally symmetric Gram operator of $p$. Regarding $P$ as a Hermitian form $P \in \Herm(S^k(\Vc))$, let $P_k=M(p)_k \in \Herm(S^k(\Vc))$ be defined by $P_k(z,z^\dg)=P(z,z^\dg) \norm{z}^{2(k-d)}$, as above. Then
\ba\label{eq:pk}
P_k=\Pi_{k}(P\otimes \I_{\Vc}^{\otimes k-d})\Pi_{k} \in \Herm(S^k(\Vc)),
\ea
which can be verified by noting that $\ip{z^{\otimes k}}{P_k z^{\otimes k}}=P(z,z^\dg)\norm{z}^{2(k-d)}$ under this choice. Since $P_k$ is real-valued, we can view it as a real operator
\ba
P_k=\Pi_{k}(P\otimes \I_{\Vr}^{\otimes k-d})\Pi_{k} \in \Sym(S^k(\Vr)).
\ea
Note that $P_k$ is a \textit{Gram operator} for $p(x) \norm{x}^{2(k-d)} \in S^k(\Vr^*)$, since
\ba
\ip{x^{\otimes k}}{P_k x^{\otimes k}}=p(x) \norm{x}^{2(k-d)}
\ea
for all $x \in \Vr$. In contrast to $P=M(p)$, $P_k$ is not in general maximally symmetric. In Appendix~\ref{ap:mk} we give a formula for $P_k$ in coordinates, which we found useful for our numerical implementation.

\begin{remark}
Let $\Vr=\real^n$ and $\Vc=\complex^n$ (the complexification of $\Vr$), and let $N^{(d)}:=M(\norm{x}^{2d})$ be the maximally symmetric Gram operator of $\norm{x}^{2d}$. Note that $N^{(1)}(z,z^\dg)=\norm{z}^{2}$, but $N^{(d)}(z,z^\dg)\neq \norm{z}^{2d}$ in general. Indeed, as a Hermitian operator $\norm{z}^{2d}=\Pi_{d}$ is the projection onto the symmetric space, which is not maximally symmetric for $d >1$.
\end{remark}

\subsection{HSOS hierarchy}

In this section, we present the HSOS hierarchy in more details. Recall that for a Hermitian operator $H \in \Herm(S^d(\Vc))$ we define
\ba
H_{\min}=\min_{\substack{z \in \Vc\\ \norm{z}=1}} H(z,z^\dg).
\ea

\begin{theorem}[HSOS hierarchy]\label{thm:hsos}
Let $H \in \Herm(S^d(\Vc))$ be a Hermitian form, and let $H_k(z,z^\dg)=H(z,z^\dg) \norm{z}^{2(k-d)}$ when $k \geq d$. For each $k\geq d$, let $\mu_k=\lambda_{\min}(H_k)$. Then $\mu_d \leq \mu_{d+1} \leq \dots $ and $\lim_k \mu_k = H_{\min}$. Furthermore,
\ba
H_{\min} - \mu_k \leq \norm{H}_{\infty} \frac{4 d(n-1)}{k+1} = O(1/k).
\ea
\end{theorem}
Similar results are proven in~\cite{quillen1968representation,catlin1995stabilization,to2006effective}. We prove this theorem using the quantum de Finetti theorem~\cite{christandl2007one}. We use the form of this theorem stated and proven in~\cite[Theorem 7.26]{Wat18}.

\begin{theorem}[Quantum de Finetti theorem]\label{thm:definetti}
Let $k \geq d$ be integers. For any unit vector $\psi \in S^k(\Vc)$, there exists an operator
\ba\label{eq:tau}
\tau \in \setft{conv}\{(uu^\dg)^{\otimes d} : u \in \Vc , \norm{u}=1\}\subseteq \End(S^d(\Vc))
\ea
for which
\ba
\norm{\tr_{[d+1 \cdot \cdot k]}(\psi \psi^\dg)-\tau}_1 \leq \frac{4 d (n-1)}{k+1}.
\ea
\end{theorem}
In the proposition statement, $\setft{conv}$ denotes the convex hull, and $\tr_{[d+1 \cdot \cdot k]}(\psi \psi^\dg)\in \Hom(S^d(\Vc))$ is the partial trace of $\psi \psi^\dg$ over $k-d$ subsystems (see Section~\ref{background}). Now we can prove Theorem~\ref{thm:hsos}.

\begin{proof}[Proof of Theorem~\ref{thm:hsos}]
Let $\psi_k \in S^k(\Vc)$ be a unit eigenvector for $H_k$ with minimum eigenvalue, and let $\tau_k \in \End(S^d(\Vc))$ be an operator of the form~\eqref{eq:tau} for which
\ba
\norm{\tr_{[d+1 \cdot \cdot k]}(\psi_k \psi_k^\dg)-\tau_k}_1 \leq \frac{4 d (n-1)}{k+1}.
\ea
Then
\ba
H_{\min}- \mu_k &\leq \tr(H \tau_k) - \mu_k\\
&= \tr(H(\tau_k - \tr_{[d+1 \cdot \cdot k]}(\psi_k \psi_k^\dg)))\\
 &\leq \norm{H}_{\infty} \norm{\tr_{[d+1 \cdot \cdot k]}(\psi_k \psi_k^\dg)-\tau_k}_1\\
&\leq \norm{H}_{\infty} \frac{4 d (n-1)}{k+1},
\ea
where the first line follows from the form of $\tau$ and convexity, the second line is straightforward, the third line follows from the inequality $\tr(AB)\leq \norm{A}_{\infty} \norm{B}_1$, and the fourth line follows from the quantum de Finetti theorem. This completes the proof.
\end{proof}

\subsection{m-HSOS hierarchy}

Let $\V_i=\complex^{n_i}$ for $i=1,\dots, m$, let $\U=S^{d_1}(\V_1)\ootimes S^{d_m}(\V_m)$, and let $H \in \Herm(\U)$ be a Hermitian operator. For $z_i, w_i \in \V_i$ and $\bfz:=(z_1,\dots, z_m)$, $\bfw=(w_1,\dots, w_m)$ we define
\ba
H(\bfz,\bfw^\dg):= \ip{w_1^{\otimes d_1}\ootimes w_m^{\otimes d_m}}{H \;\; z_1^{\otimes d_1}\ootimes z_m^{\otimes d_m}}.
\ea
By similar arguments as in Proposition~\ref{prop:polarization}, $H$ is uniquely determined by $H(\bfz,\bfz^\dg)$. We call $H$ a \textit{multi-homogeneous Hermitian form} of degree $\bfd=(d_1,\dots, d_m)$. Let

\ba
H_{\min}=\min_{\substack{z_i \in \V_i \\ \norm{z_i}=1}} H(\bfz,\bfz^\dg).
\ea
For $k \geq \max_j d_j$, let
\ba
H_k(\bfz,\bfz^\dg)=H(\bfz,\bfz^\dg) \cdot (\norm{z_1}^{2(k-d_1)}\ootimes \norm{z_n}^{2(k-d_m)})\in \Herm(S^k(\V_1)\ootimes S^k(\V_m)).
\ea
As an operator,
\ba
H_k=(\Pi_{d_1} \ootimes \Pi_{d_m})(H \otimes \I_{\V_1}^{\otimes k-d_1} \ootimes \I_{\V_m}^{\otimes k-d_m})(\Pi_{d_1} \ootimes \Pi_{d_m}),
\ea
where $\Pi_{d_i}$ is the orthogonal projection onto $S^{d_i}(\V_i)$. This can be verified by checking that
\ba
\ip{z_1^{\otimes d_1}\ootimes z_m^{\otimes d_m}}{H_k \;\; z_1^{\otimes d_1}\ootimes z_m^{\otimes d_m}}=H_k(\bfz,\bfz^\dg)
\ea
under this choice.

The following theorem shows that the minimum eigenvalue of $H_k$ converges to $H_{\min}$.

\begin{theorem}[m-HSOS hierarchy]\label{thm:mhsos}
Let $H \in \Herm(\U)$, and let $\mu_k=\lambda_{\min}(H_k)$
when $k \geq d:= \max_i d_i$. Then $\mu_d \leq \mu_{d+1} \leq \dots $ and $\lim_k \mu_k = H_{\min}$. Furthermore,
\ba
H_{\min} - \mu_k \leq \norm{H}_{\infty} \frac{4 |\bfd| (\max_j n_j-1)}{k+1} = O(1/k),
\ea
where $|\bfd|=d_1+\dots + d_m$.
\end{theorem}

We defer the proof to Appendix~\ref{app:mhsos_proof}, as it is similar to Theorem~\ref{thm:hsos}.

\section{Proof of convergence}\label{sec:converge}

In this section we prove our reduction from real to Hermitian optimization on the sphere (Theorem~\ref{thm:intro_reduction} from the introduction). We then use this result along with the HSOS hierarchy (Theorem~\ref{thm:hsos}) to obtain our hierarchy of eigencomputations for the optimization problem~\eqref{eq:optimization_problem} (Theorem~\ref{thm:intro_hrsos} from the introduction).

\subsection{Reduction from real to Hermitian optimization on the sphere}
In this section we prove Theorem~\ref{thm:intro_reduction}. We first restate the theorem.
\begin{theorem}\label{thm:real_opt}
Let $\Vr=\real^n$, let $\Vc=\complex^n$ (the complexification of $\Vr$), let $p \in S^{2d}(\Vr^*)$, and let $P=M(p) \in \Herm(S^d(\Vc))$ be the maximally symmetric Gram operator of $p$. Then
\ba
P_{\min} \leq p_{\min} \leq \frac{P_{\min}}{\delta(d)},
\ea
where $\delta(d)={2^d}{\binom{2d}{d}^{-1}}$.
\end{theorem}

\begin{proof}[Proof of Theorem~\ref{thm:real_opt}]
%
Let $z \in \Vc$ be a unit vector for which $P_{\min}=P(z,z^\dg)$. By definition, $P(z,z^\dg)=\tr((zz^\dg)^{\otimes d} P)$. Let $x,y \in \Vr$ be such that $z=x+iy$. Recall that $P$ is invariant under partial transposition along any of the $d$ factors of $\Vc$. In particular, $P=\frac{1}{2}(P+P^{\t_1})$, where $(\cdot)^{\t_1}$ denotes the partial transpose on the first factor. Thus,
\ba
\tr((zz^\dg)^{\otimes d}P)&= \frac{1}{2} \tr((zz^\dg)^{\otimes d}P)+ \frac{1}{2} \tr((zz^\dg)^{\otimes d}P^{\t_1})\\
					&=\frac{1}{2} \tr((zz^\dg)^{\otimes d}P)+ \frac{1}{2} \tr((zz^\dg)^{\t}\otimes (zz^\dg)^{\otimes d-1}P)\\
					&=\tr\bigg((xx^\t+yy^\t)\otimes (zz^\dg)^{\otimes d-1} P\bigg),
\ea
where the last line follows from $zz^\dg+(zz^\dg)^{\t}=2(xx^\t+yy^\t)$. Continuing in this way for the other factors, we obtain
\ba
\tr((zz^\dg)^{\otimes d}M(p))&=\tr\left[(xx^\t+yy^\t)^{\otimes d} \; P\right]\\
&=\tr\bigg[M\big(((x^{\t})^2+(y^{\t})^2)^d\big) \bigg],
\ea
where we recall that we associate the row vector $x^{\t}$ with the linear form $x^{\t}(a) = \ip{x}{a}$, so $(x^{\t})^2$ is the polynomial $(x^{\t})^2(a) = \ip{x}{a}^2$. Let $Q=M(((x^{\t})^2+(y^{\t})^2)^d)$, and let $\delta=\tr(Q).$ By~\cite[Corollary 5.6]{reznick2013length} there exist real unit vectors ${v_1},\dots, {v_{d+1}} \in \Vr$ for which
\ba
\frac{1}{\delta} Q \in \setft{conv}\{(v_1v_1^\t)^{\otimes d},\dots, (v_{d+1}v_{d+1}^\t)^{\otimes d}\}.
\ea
Multiplying both sides by $P$ and taking the trace, we obtain
\ba
\frac{\Pmin}{\delta} \in \setft{conv}\{\tr(({v_i}{v_i}^\t)^{\otimes d} P) : i \in [d+1]\}.
\ea
Thus, there exists $i \in [d+1]$ for which $\tr(({v_i}{v_i}^\t)^{\otimes d} P) \leq \frac{\Pmin}{\delta}$. Recall that $ \tr(({v_i}{v_i}^\t)^{\otimes d} P)= p(v_i) \geq p_{\min}$.

To complete the proof, it suffices to show that $\delta \geq \delta(d)$. Let $e_1,e_2 \in \Vr$ be orthonormal, and let $U \in \Orth(\Vr)$ be an orthogonal operator for which $U(xx^\t+yy^\t)U^{\t}=t e_1 e_1^\t+(1-t)e_2 e_2^\t$ for some $t \in [0,1]$. Redefine $Q \mapsto U^{\otimes d} Q (U^{\t})^{\otimes d}$, which has the same trace and is given by
\ba
Q&=M((t (e_1^{\t})^2 + (1-t) (e_2^{\t})^2)^d)\\
&=\sum_{j=0}^d \binom{d}{j} t^{2j} (1-t)^{2(d-j)} M((e_1^{\t})^{2j} (e_2^{\t})^{2(d-j)}).
\ea
Note that $\tr(M((e_1^{\t})^{2j} (e_2^{\t})^{2(d-j)}))=\binom{d}{j} \binom{2d}{2j}^{-1}$ by Remark~\ref{rmk:trace}. Hence,
\ba
\delta=\sum_{j=0}^d \binom{d}{j}^2 \binom{2d}{2j}^{-1}\;\; t^j (1-t)^{d-j}.
\ea
To complete the proof, it suffices to prove the following:
\begin{lemma}\label{lemma:technical} It holds that
\ba
\min_{t \in [0,1]} \delta= {2^d}{\binom{2d}{d}^{-1}}.
\ea
\end{lemma}
The proof of this lemma is quite technical, and we defer it to Appendix~\ref{ap:lemma}. Alternatively, one can easily see that $\delta \geq \min_{j} \binom{d}{j}\binom{2d}{2j}^{-1}$, and this expression can be substituted for $\delta(d)$ in the Theorem statement for an easier but slightly weaker bound.
\end{proof}

\begin{remark}\label{rmk:nec} Theorem~\ref{thm:real_opt} does not hold if we replace $P$ by an arbitrary Gram operator of $p$. For example, let $\{e_1, e_2\}\subseteq \real^2$ be an orthonormal basis, let
\ba
\U=\spn_{\real}\{e_1 \otimes e_2 + e_2 \otimes e_1,\;\; e_1 \otimes e_1 - e_2 \otimes e_2\} \subseteq S^2(\real^2),
\ea
and let $\Pi_{\U} \in \End(S^2(\real^2))$ be the orthogonal projection onto $\U$, which is a Gram operator for $\frac{1}{2}\norm{x}^4$. Viewing $\Pi_\U$ as a complex operator, we have
\ba
\max_{\substack{v \in \complex^2 \\ \norm{v}=1}} \ip{v \otimes v}{\Pi_{\U}\;  v \otimes v} = 1
\ea
because $\frac{1}{2}(e_1+i e_2)^{\otimes 2} \in \U \otimes_{\real} \complex$. Let $Q=\frac{3}{4}\Pi_{S^2(\real^2)}-\Pi_{\U}$, which is a Gram operator for $p(x):=\frac{1}{4}\norm{x}^4$. Then $p_{\min}=1/4$, but when $Q$ is viewed as a Hermitian operator, $Q_{\min}=-1/4$.
\end{remark}

\subsection{HRSOS hierarchy}
For a real form $p \in S^{2d}(\Vr^*)$ recall that we define $P=M(p) \in \Herm(S^d(\Vc))$ to be the maximally symmetric Gram operator of $P$, and
\ba
P_k(z,z^\dg) = P(z,z^\dg) \norm{z}^{2(k-d)} \in \Herm(S^k(\Vc)),
\ea
when $k \geq d$. Similarly, we let $N^{(d)}=M(\norm{x}^{2d}) \in \Herm(S^d(\Vc))$, and let
\ba
N^{(d)}_k (z,z^\dg)=N^{(d)} (z,z^\dg) \norm{z}^{2(k-d)} \in \Herm(S^k(\Vc))
\ea
when $k \geq d$. As operators, these are given by $P_k=\Pi_k(P \otimes \I_{\Vc}^{\otimes k-d})\Pi_k$ and ${N_k^{(d)} = \Pi_k(N^{(d)} \otimes \I_{\Vc}^{\otimes k-d})\Pi_k}$. We also recall that $\kappa(\cdot)$ is the condition number, and
\ba\label{eq:delta}
\delta(d)={2^d}{\binom{2d}{d}^{-1}}.
\ea
\begin{theorem}[HRSOS hierarchy]\label{thm:convergence}
Let $p \in S^{2d}(\Vr^*)$ be a real form of even degree, let $P=M(p)$, and for each $k \geq d$ let
\ba
\eta_k=\lambda_{\min}((N_k^{(d)})^{-1/2}P_k (N_k^{(d)})^{-1/2}).
\ea
Then $\eta_d \leq \eta_{d+1} \leq \dots$ and $\lim_k \eta_k = \pmin$. Furthermore,
\ba
p_{\min}-\eta_k \leq \norm{P}_{\infty}\left(1+\kappa(N^{(d)})\right) \frac{4 d (n-1)}{\delta(d)(k+1)} = O\left(\frac{1}{k}\right).
\ea
\end{theorem}

Numerics suggest that
\ba
\kappa(N^{(d)})=\binom{n/2+d-1}{\floor{d/2}},
\ea
where the binomial function is extended to non-integer inputs via the Gamma function. However we do not have a proof of this.

Modulo this unproven expression for $\kappa(N^{(d)})$, this gives an explicit dependence on $k,n,d,$ and $\norm{P}_{\infty}$ for our HRSOS hierarchy, which scales as $\mathcal{O}( \norm{P}_{\infty} n^{\frac{d}{2}+1}k^{-1})$ for fixed $d$. In particular, this also upper bounds the additive error of the RSOS hierarchy. For reference, we now review the scaling of the upper bounds known for the RSOS hierarchy, for fixed $d$. In~\cite[Theorem 1]{fang2021sum} a bound of ${ (p_{\max}-p_{\min}) C_d n^2 k^{-2}}$ is proven for an unspecified constant $C_d$. Earlier bounds for the RSOS hierarchy scale as ${\mathcal{O}((p_{\max}-p_{\min}) nk^{-1})}$ with an explicit constant~\cite{reznick1995uniform,faybusovich2004global,doherty2012convergence}, see in particular~\cite[Theorem 1]{faybusovich2004global} and the discussion following~\cite[Corollary 7.1]{doherty2012convergence}. Note that $p_{\max}-p_{\min} \leq \norm{P}_{\infty}$, so a bound in terms of $p_{\max}-p_{\min}$ is a priori more desirable than a bound in terms of $\norm{P}_{\infty}$.

\paragraph{{Note added:}} In a recent follow up work~\cite[Theorem 5.6]{blomenhofer2024moment} an upper bound for the additive error of the HRSOS hierarchy is proven which is $(p_{\max}-p_{\min})C_{n,d}$ for an unspecified term $C_{n,d}$ that depends on $n$ and $d$. This work also provides an alternate proof of a similar statement as Theorem~\ref{thm:convergence}, with similar dependence on $k,n,d,$ and $\norm{P}_{\infty}$~\cite[Theorem 5.4]{blomenhofer2024moment}.\footnote{We note that their bound does not rely on any unproven numerical expressions, such as our expression for $\kappa(N^{(d)})$.} This work furthermore proves a lower bound for the HRSOS hierarchy which scales as $\Omega(1/k^2)$ in the parameter $k$~\cite[Theorem 5.8]{blomenhofer2024moment}.



\begin{proof}[Proof of Theorem~\ref{thm:convergence}]
First note that
\ba
\eta_k&=\lambda_{\min}((N_k^{(d)})^{-1/2}P_k (N_k^{(d)})^{-1/2})\\
	&=\max\{ \eta : (N_k^{(d)})^{-1/2}P_k (N_k^{(d)})^{-1/2} - \eta \; \I_{S^k(\Vr)} \succeq 0 \}\\
	&=\max\{\eta : P_k - \eta \; N_k^{(d)} \succeq 0\}\\
	&=\max\{ \eta : M(p-\eta \cdot \norm{x}^{2d})_k \succeq 0 \}.
\ea
Let
\ba
q^{(k)}(x) = p(x)-\eta_k \norm{x}^{2d} \in S^{2d}(\Vr^*),
\ea
and let $Q^{(k)}=M(q^{(k)}).$ For the inequality $\eta_k \leq p_{\min}$, note that since $Q^{(k)}_k \succeq 0$, we have
\ba
p(x)-\eta_k&=\ip{x^{\otimes k}}{Q_k^{(k)}  x^{\otimes k}}\geq 0
\ea
for all $x \in \Vr$, so $p_{\min} - \eta_k \geq 0$. For the inequality $\eta_{k+1} \geq \eta_k$, let $\psi \in S^{k+1}(\Vc)$ be a unit eigenvector for $P_{k+1}$ with eigenvalue $\eta_{k+1}$. Then
\ba
\eta_{k+1}=\tr(P_{k+1} \psi \psi^\dg) = \tr(P_k \tr_{1}(\psi \psi^\dg))\geq \eta_k,
\ea
where the second equality follows from $\tr_1(P_{k+1})=P_k$, and the inequality follows by convexity.

For the bound, note that
\ba\label{eq:qmin}
Q_{\min}^{(k)} &\leq \norm{Q^{(k)}}_{\infty} \frac{4d (n-1)}{k+1}\\
		&\leq \norm{P}_{\infty} \left(1+\kappa(N^{(d)})\right) \frac{4d(n-1)}{k+1}.
\ea
The first line follows from Theorem~\ref{thm:hsos} and the fact that $\lambda_{\min}(Q_k^{(k)}) = 0$. The second line follows from
\ba
\norm{Q^{(k)}}_{\infty}&\leq \norm{P}_{\infty}+ \abs{\eta_k} \norm{N^{(d)}}_{\infty} \\
			&\leq \norm{P}_{\infty}\left(1+\frac{\lambda_{\max}(N^{(d)})}{\lambda_{\min}(N^{(d)})}\right)
\ea
Here, the first line is the triangle inequality. The second line follows from $\norm{N^{(d)}}_{\infty}=\lambda_{\max}(N^{(d)})$ by Proposition~\ref{prop:Mspos}, and $\abs{\eta_k}\leq \norm{P}_{\infty} \;\lambda_{\min}(N^{(d)})^{-1}$ since choosing $\eta$ equal to minus the righthand side would guarantee $M(p-\eta \; \norm{x}^{2d}) \succeq 0$.

It follows from Theorem~\ref{thm:real_opt} and~\eqref{eq:qmin} that
\ba
p_{\min}-\eta_k &= q_{\min}^{(k)}\\
 &\leq \frac{Q_{\min}^{(k)}}{\delta(d)} \\
 &=\norm{P}_{\infty}\left(1+\kappa(N^{(d)})\right) \frac{4d(n-1)}{\delta(d) (k+1)}.
\ea
This completes the proof.
\end{proof}

\section{Numerical implementation and examples}\label{sec:numerics}

Let $\Vr=\real^n$. In this section we implement our polynomial optimization hierarchy described in Theorem~\ref{thm:convergence}. While the quantity $\eta_k=\lambda_{\min}((N_k^{(d)})^{-1/2}P_k (N_k^{(d)})^{-1/2})$ that defines the $k$-th level of our hierarchy is the minimum eigenvalue of a matrix, it should not be computed as such, since doing so would require computation of the inverse of the matrix $N_k^{(d)}=M(\norm{x}^{2d})_k$. 
An alternate way of computing $\eta_k$ is to instead solve the generalized eigenvalue/eigenvector problem
\begin{align}\label{eq:generalized_eigenvector}
    P_k \; \psi  = \lambda \; N_k^{(d)} \; \psi
\end{align}
in $\psi \in S^k(\Vr)$ and $\lambda \in \real$.

Indeed, the minimum generalized eigenvalue $\lambda$ (i.e., the minimal $\lambda$ for which there is a nonzero vector $\psi$ solving Equation~\eqref{eq:generalized_eigenvector}) is equal to $\eta_k$, since
\[
    P_k \; \psi  = \lambda \; N_k^{(d)} \; \psi \quad \Longleftrightarrow \quad (N_k^{(d)})^{-1}P_k\; \psi = \lambda \; \psi,
\]
and is it a standard fact that the matrices $AB$ and $BA$ have the same eigenvalues as each other (see \cite[Theorem~1.3.22]{HJ13}, for example, and apply this fact to the matrices $A = (N_k^{(d)})^{-1/2}$ and $B = (N_k^{(d)})^{-1/2}P_k$).

However, this generalized eigenvalue can be found without inverting or multiplying any matrices. Furthermore, there are extremely fast numerical algorithms for solving this problem that can exploit the extreme sparsity of $P_k$ and $N_k^{(d)}$ \cite{Ste02,magron2023sparse}; $N_k^{(d)}$ is sparse because the form $\norm{x}^{2d}$ itself is sparse (i.e., most of its coefficients are equal to $0$), and $P_k$ is sparse when $k$ is large even if $p$ is dense since
\ba
P_k= \Pi_{k} (P \otimes \I_{\Vr}^{\otimes {k-d}})\Pi_{k},
\ea
and $\I_{\Vr}^{\otimes k-d}$
is sparse. These generalized eigenvalue algorithms have been implemented in ARPACK \cite{arpack}, making them available out-of-the-box in SciPy, Mathematica, MATLAB, and many other popular computational software packages. We have implemented the computation of $\eta_k$ in the QETLAB package for MATLAB \cite{qetlab}.

We will now present several examples to illustrate how well our HRSOS hierarchy performs numerically compared to the RSOS hierarchy~\cite{reznick1995uniform, doherty2012convergence,fang2021sum}, the DSOS hierarchy~\cite{AM19}, and the harmonic hierarchies of~\cite{cristancho2024harmonic}. In all of these hierarchies, the $k$-th level performs computations in the space of homogeneous polynomials of degree $2k$, and the first non-trivial level is $k=d$. These examples all support the following conclusions:
\begin{itemize}
    \item If the polynomial is small enough to be bounded by the RSOS hierarchy, then the RSOS hierarchy should be used, since it typically provides better bounds per unit of computation time. However, since it relies on semidefinite programming, it can only be applied to comparably small problems, e.g. degree-4 polynomials with $25$ or fewer variables.

    \item If the polynomial is \emph{even} (i.e., in each term, each variable is raised to an even power), specialized optimization techniques have been developed \cite{polya1974george,de2005equivalence,de2015error,AM19,ahmed2019maximization,ahmed2021two,vargas2023copositive}. For example, the diagonally-dominant sum-of-squares (DSOS) hierarchy of \cite{AM19} performs quite well when applied to even polynomials, and has a small enough memory footprint that it can be applied to larger problems, e.g. degree-4 polynomials with $70$ or fewer variables.

    \item For larger problems (e.g., a degree-4 polynomial with more than $25$ variables, or a degree-4 even polynomial with more than $70$ variables), our hierarchy outperforms existing hierarchies.
    For example, our hierarchy provides better bounds per unit of computation time than the optimization-free hierarchies of \cite{cristancho2024harmonic} for all problem sizes. Furthermore, our hierarchy provides significantly better bounds (at the expense of somewhat higher computation time) than the DSOS hierarchy, and can be applied to much larger problems, e.g. degree-4 polynomials with $90$ variables.
\end{itemize}

\begin{example}[Homogeneous Motzkin polynomial]\label{exam:motzkin}
Let
\ba
p(x) = x_1^2 x_2^2 (x_1^2 + x_2^2 - 3x_3^2) + x_3^6
\ea
be the homogeneous Motzkin polynomial of degree $2d=6$~\cite[Section~3.2]{Lau09}. This polynomial is non-negative but not a sum of squares; in fact, its minimum value on the unit sphere is exactly $0$. Because this polynomial has so few non-zero coefficients, the matrix $P_k$ is extremely sparse, so our hierarchy can be run at extremely high levels---on standard desktop hardware\footnote{A 2.10GHz Core~i7 with 16Gb of RAM.} we have been able to go up to level $k=d+2000$. Lower bounds on the minimum value of this polynomial at various levels of the hierarchy, and the time required to compute those bounds, are illustrated and compared against numerics provided by the optimization-free hierarchies of \cite{cristancho2024harmonic} in Table~\ref{tab:motzkin_numerics} and Figure~\ref{fig:motzkin_numerics}. In particular, our hierarchy outperforms those optimization-free hierarchies both per level of the hierarchy and per unit of real computation time.

We note that the Motzkin polynomial is even and has just $3$ variables, so the RSOS and DSOS hierarchies both perform extremely well on it: they can compute the exact minimum value (which is $0$) in less than a second.
\end{example}

\begin{table}[!htb]
    \begin{center}
        \begin{tabular}{ @{ \ }ccc@{ \qquad }cc@{ \qquad }cc@{ \ } }
            \toprule
            & \multicolumn{2}{c}{\textbf{our hierarchy (HRSOS)}}& \multicolumn{2}{c}{\textbf{squares} \cite{cristancho2024harmonic}} & \multicolumn{2}{c}{\textbf{Fawzi} \cite{cristancho2024harmonic}} \\
            $k$ & lower bound & time & l.b. & time & l.b. & time \\ \midrule
            $10$ & $-0.028748$ & $0.010$ s & $-0.436125$ & $0.002$ s & $-0.070812$ & $0.013$ s \\
            $20$ & $-0.010490$ & $0.059$ s & $-0.120558$ & $0.015$ s & $-0.022299$ & $0.092$ s \\
            $40$ & $-0.004682$ & $0.361$ s & $-0.045407$ & $0.021$ s & $-0.008391$ & $1.019$ s \\
            $80$ & $-0.002225$ & $2.488$ s & $-0.019759$ & $0.082$ s & $-0.003561$ & $13.287$ s \\
            $160$ & $-0.001086$ & $16.586$ s & $-0.009222$ & $0.326$ s & $-0.001635$ & $202.372$ s \\\bottomrule
        \end{tabular}
        \caption{Lower bounds on the minimum value of the Motzkin polynomial, as computed by the $k$-th level of our HRSOS hierarchy, the squares-based hierarchy of \cite{cristancho2024harmonic}, and the Fawzi hierarchy of \cite{cristancho2024harmonic}, as well as the time required to compute those bounds.}\label{tab:motzkin_numerics}
    \end{center}
\end{table}


\begin{example}[Random low-variable quartic polynomials]\label{exam:random_poly}
    Our hierarchy performs quite well on randomly-generated polynomials. To illustrate this fact, we generated dense random\footnote{All coefficients non-zero, independently chosen from a standard normal distribution.} degree-$4$ polynomials with $3$, $6$, and $10$ variables, and applied as many levels as possible of our hierarchy and the optimization-free hierarchies of \cite{cristancho2024harmonic} to all three of those polynomials.

Lower bounds on the minimum value of this polynomial at various levels of the hierarchy, and the time required to compute those bounds, are illustrated in Table~\ref{tab:quartic_numerics} and Figures~\ref{fig:quartics_graph}(b--d). In all cases, our hierarchy provides better bounds both per level of the hierarchy and per unit of real computation time, with the improvement provided by our hierarchy getting more pronounced as the problem size increases. For example, the best lower bound on the random 6-variable quartic that could be computed in less than 6~hours by the optimization-free hierarchies of~\cite{cristancho2024harmonic} (at level~$k = 16$ of those hierarchies) was beaten in less than a second by the $k = 8$ level of our hierarchy. Furthermore, those optimization-free hierarchies were unable to compute any bounds at all on the 10-variable quartic due memory limitations. 
     By contrast, our hierarchy could be run up to $k=11$ in this case.
     
The two main bottlenecks in the hierarchies of~\cite{cristancho2024harmonic} are constructing a degree-$2k$ cubature rule (which does not depend on $p$, but depends on a choice of cubature rule) and minimizing a certain degree-$2k$ polynomial (which depends on $p$ and a choice of kernel) at the points of that cubature rule. The cubature rule they consider has size $2(k+1)^{n-1}$~\cite[Corollary 2.3]{cristancho2024harmonic}. By storing the entire cubature rule, their Julia implementation runs into memory issues already for problems as small as degree-$4$ polynomials in $10$ variables. It is possible that this can be avoided by generating elements sequentially, but we have not attempted to do so. When we do not encounter memory issues, we find that the second bottleneck takes the majority of the time, and even if we only account for the time taken in the second bottleneck, we obtain similar plots as in Figures~\ref{fig:quartics_graph}(a--c).

    We did not include bounds from the DSOS hierarchy in Figures~\ref{fig:quartics_graph}(b--d) since it did not converge for the polynomials that we generated. For example, for the $3$-variable quartic polynomial used to generate Figure~\ref{fig:quartics_graph}(b), the first 10 non-trivial levels of the DSOS hierarchy all produced a lower bound with $\log_{10}(1/\text{error}) \approx -0.0931$; a worse bound than the one produced by the first non-trivial level ($k=2$) of each of the other hierarchies.

\end{example}

\begin{table}[!htb]
    \begin{center}
        \begin{tabular}{ @{ \ }ccc@{ \qquad }cc@{ \qquad }cc@{ \ } }
            \toprule
            & \multicolumn{2}{c}{\textbf{our hierarchy (HRSOS)}}& \multicolumn{2}{c}{\textbf{squares} \cite{cristancho2024harmonic}} & \multicolumn{2}{c}{\textbf{Fawzi} \cite{cristancho2024harmonic}} \\
            $k$ & error & time & error & time & error & time \\ \midrule
            $6$ & $0.337860$ & $0.106$ s & $0.635424$ & $10.735$ s & $0.769963$ & $14.998$ s\\
$8$ & $0.184064$ & $0.668$ s & $0.467328$ & $36.075$ s & $0.574778$ & $47.137$ s\\
$10$ & $0.126534$ & $3.145$ s & $0.368130$ & $107.739$ s & $0.456129$ & $124.972$ s\\
$12$ & $0.096335$ & $15.033$ s & $0.298842$ & $230.852$ s & $0.375753$ & $283.068$ s\\
$14$ & $0.077792$ & $70.045$ s & $0.251596$ & $562.863$ s & $0.310968$ & $1003.156$ s\\
$16$ & $0.065239$ & $233.127$ s & $0.214861$ & $10698.535$ s & $0.268982$ & $18931.345$ s\\
\bottomrule
        \end{tabular}
        \caption{The error when computing a lower bound on the minimum value of a random dense 6-variable quartic homogeneous polynomial, as well as the time required to compute these bounds, when using the $k$-th level of our HRSOS hierarchy, the squares-based hierarchy of \cite{cristancho2024harmonic}, and the Fawzi hierarchy of \cite{cristancho2024harmonic}.}\label{tab:quartic_numerics}
    \end{center}
\end{table}

\begin{figure}[!htb]
    \centering

    \begin{subfigure}[t]{0.47\textwidth}
        \centering

        \caption{The Motzkin polynomial.}\label{fig:motzkin_numerics}
    \end{subfigure} \ \ \ %
    ~
    \begin{subfigure}[t]{0.47\textwidth}
        \centering
    	%
        \caption{A random dense 3-variable quartic.}
    \end{subfigure}\\[0.4cm]

    \begin{subfigure}[t]{0.5\textwidth}
        \centering
    	%
        \caption{A random dense 6-variable quartic.}
    \end{subfigure}%
    ~
    \begin{subfigure}[t]{0.5\textwidth}
        \centering
    	%
        \caption{A random dense 10-variable quartic.}
    \end{subfigure}

	\caption{A plot of the number of decimal places of accuracy versus (the logarithm of) the number of seconds required to achieve that accuracy for our HRSOS hierarchy (blue, solid), the squares-based hierarchy of \cite{cristancho2024harmonic} (red, long dashes), and the Fawzi hierarchy of \cite{cristancho2024harmonic} (green, short dashes). In all cases, our hierarchy outperforms the others. The exact minimum value of each of these polynomials (and thus the number of decimal places of accuracy) was found by using the RSOS hierarchy to get a lower bound and then finding a point on the sphere for which that lower bound is attained.}\label{fig:quartics_graph}
\end{figure}

\begin{example}[Random high-variable quartic polynomials]\label{exam:random_poly}
    Not only does our hierarchy produce better bounds than other non-RSOS hierarchies, but it can also be applied to much larger polynomials since it is less memory-intensive. It was noted in \cite{AM19} that while the RSOS hierarchy can only be used for quartic polynomials with up to $25$ variables or so, the DSOS hierarchy can be used for much larger quartic polynomials with as many as $70$ variables.\footnote{On our hardware, we were able to run the RSOS hierarchy up to $23$ variables and the DSOS hierarchy up to $58$ variables.} Our hierarchy can go even farther, producing bounds on dense quartic polynomials with more than $90$ variables, as illustrated in Figure~\ref{fig:DSOS_bounds}. Furthermore, our hierarchy runs in time comparable to the DSOS hierarchy, while producing significantly better bounds (see Figure~\ref{fig:DSOS_timings}).

    We did not include the results of the optimization-free hierarchies from \cite{cristancho2024harmonic} in Figure~\ref{fig:dsos_compare}, since these were unable to produce bounds for any quartic polynomials with $n \geq 10$ variables due to memory limitations.

    \begin{figure}[!htb]
        \centering

        \begin{subfigure}[t]{\textwidth}
            \centering
        	\begin{tikzpicture}[xscale=0.15,yscale=2]
        		\foreach \x in {5,10,15,...,90} \draw[color=gray!25] (\x,-0.5) -- (\x,-3.2);
        		\foreach \y in {-1,-2,-3} \draw[color=gray!25] (5,\y) -- (93,\y);

        		\foreach \x in {5,10,15,...,90} \draw (\x,-0.53) -- (\x,-0.44) node[anchor=south] {\footnotesize $\x$};
        		\draw (5.6,-1) -- (3.8,-1) node[anchor=east] {\footnotesize $-1$};
        		\draw (5.6,-2) -- (3.8,-2) node[anchor=east] {\footnotesize $-10$};
        		\draw (5.6,-3) -- (3.8,-3) node[anchor=east] {\footnotesize $-100$};

            \begin{scope}
                \clip (0,0) rectangle (95,-3.2);
        		\draw[color=blue] (5,-1.293087304) -- (6,-1.378691087) -- (7,-1.418903747) -- (8,-1.448625945) -- (9,-1.479068746) -- (10,-1.505849172) -- (11,-1.517603996) -- (12,-1.570380516) -- (13,-1.581293457) -- (14,-1.757789373) -- (15,-1.765139183) -- (16,-1.769934461) -- (17,-1.778044401) -- (18,-1.79978964) -- (19,-1.807709252) -- (20,-1.820036054) -- (21,-1.828592492) -- (22,-1.837711639) -- (23,-1.848097197) -- (24,-1.8570311) -- (25,-1.86379527) -- (26,-1.8738084) -- (27,-1.884467453) -- (28,-1.892316593) -- (29,-1.899125732) -- (30,-1.906625319) -- (31,-1.916756796) -- (32,-1.921881838) -- (33,-1.932382178) -- (34,-1.942730862) -- (35,-1.951089307) -- (36,-1.962491129) -- (37,-1.972061739) -- (38,-1.979077361) -- (39,-1.98581833) -- (40,-1.993187664) -- (41,-2.000418632) -- (42,-2.004591627) -- (43,-2.008938662) -- (44,-2.015070828) -- (45,-2.02118934) -- (46,-2.029957317) -- (47,-2.037031389) -- (48,-2.045522395) -- (49,-2.05091288) -- (50,-2.05623991) -- (51,-2.062383668) -- (52,-2.069777101) -- (53,-2.07637855) -- (54,-2.081584449) -- (55,-2.087720563) -- (56,-2.094135757) -- (57,-2.100966395) -- (58,-2.105655028) -- (59,-2.113861697) -- (60,-2.132391175) -- (61,-2.120073711) -- (62,-2.125312069) -- (63,-2.134850475) -- (64,-2.133335157) -- (65,-2.137914818) -- (66,-2.155999394) -- (67,-2.157385559) -- (68,-2.163164417) -- (69,-2.169230892) -- (70,-2.174309666) -- (71,-2.178540161) -- (72,-2.18361862) -- (73,-2.191298901) -- (74,-2.189670507) -- (75,-2.198103889) -- (76,-2.195344646) -- (77,-2.209806055) -- (78,-2.208435471) -- (79,-2.219103497) -- (80,-2.220980797) -- (81,-2.226942953) -- (82,-2.229390791) -- (83,-2.236581973) -- (84,-2.234338999) -- (85,-2.249053071) -- (86,-2.249713749) -- (87,-2.253976303) -- (88,-2.265384915) -- (89,-2.263964091) -- (90,-2.267738266) -- (91,-2.273125127) -- (92,-2.274340099) -- (93,-2.279877562);

        		\draw[color=Maroon,dash pattern={on 10pt off 3pt}] (5,-1.356136808) -- (6,-1.406734292) -- (7,-1.504640605) -- (8,-1.579160165) -- (9,-1.661342466) -- (10,-1.725332377) -- (11,-1.797601134) -- (12,-1.887868129) -- (13,-1.93188636) -- (14,-1.995523843) -- (15,-2.039700976) -- (16,-2.085382388) -- (17,-2.137309318) -- (18,-2.175436835) -- (19,-2.218791202) -- (20,-2.255645621) -- (21,-2.286635136) -- (22,-2.322384481) -- (23,-2.355429699) -- (24,-2.389936212) -- (25,-2.422443688) -- (26,-2.453904456) -- (27,-2.481526326) -- (28,-2.511854002) -- (29,-2.536611009) -- (30,-2.566441511) -- (31,-2.592006449) -- (32,-2.617126174) -- (33,-2.640235052) -- (34,-2.664088706) -- (35,-2.687971725) -- (36,-2.710888184) -- (37,-2.734450614) -- (38,-2.756086169) -- (39,-2.774803075) -- (40,-2.794455707) -- (41,-2.814672058) -- (42,-2.833574436) -- (43,-2.851354874) -- (44,-2.871355197) -- (45,-2.889191097) -- (46,-2.906984623) -- (47,-2.923605152) -- (48,-2.940416865) -- (49,-2.957354969) -- (50,-2.974308687) -- (51,-2.990555406) -- (52,-3.006234435) -- (53,-3.021786539) -- (54,-3.037377388) -- (55,-3.051921685) -- (56,-3.066774784) -- (57,-3.081560114) -- (58,-3.096053818);

        		\draw[color=ForestGreen,dashed] (5,-1.088561312) -- (6,-1.189327739) -- (7,-1.191121955) -- (8,-1.19568854) -- (9,-1.20336909) -- (10,-1.283481243) -- (11,-1.290760946) -- (12,-1.307909384) -- (13,-1.309554372) -- (14,-1.51340167) -- (15,-1.513678294) -- (16,-1.514758208) -- (17,-1.516108405) -- (18,-1.542596896) -- (19,-1.544100181) -- (20,-1.55379349) -- (21,-1.559713602) -- (22,-1.569025163) -- (23,-1.569278963);
    	 \end{scope}

        	\draw[thick,-to] (5,-0.5) --node[anchor=south,shift={(0,0.75)}]{number of variables ($n$)} (93,-0.5);
        	\draw[thick,-to] (5,-0.5) --node[anchor=south,shift={(-1.1,0)},rotate=90]{lower bound} (5,-3.2);
            
            \node[draw=black,fill=white,thick,anchor=north east,shift={(-0.1,-0.1)}] at (92,-0.5) {%
                \begin{tabular}{@{}r@{ }l@{}}
                 \raisebox{2pt}{\tikz{\draw[blue] (0,0) -- (5mm,0);}}&ours\\
                 \raisebox{2pt}{\tikz{\draw[Maroon,dash pattern={on 7pt off 3pt}] (0,0) -- (5mm,0);}}&DSOS\\
                 \raisebox{2pt}{\tikz{\draw[ForestGreen,dashed] (0,0) -- (5mm,0);}}&RSOS
                \end{tabular}};
    	\end{tikzpicture}
            \caption{Lower bounds computed by level $k=d+1$ of our hierarchy, as well as the DSOS and RSOS hierarchies (higher is better).}\label{fig:DSOS_bounds}
        \end{subfigure}\\[0.5cm]

        \begin{subfigure}[t]{\textwidth}
            \centering
        	\begin{tikzpicture}[xscale=0.15,yscale=1]
        		\foreach \x in {5,10,15,...,90} \draw[color=gray!25] (\x,-1) -- (\x,5.5);
        		\foreach \y in {0,1,2,3,4,5} \draw[color=gray!25] (5,\y) -- (93,\y);

        		\foreach \x in {5,10,15,...,90} \draw (\x,-1.12) node[anchor=north] {\footnotesize $\x$} -- (\x,-0.94);
        		\draw (5.6,-1) -- (3.8,-1) node[anchor=east] {\footnotesize $0.01$};
        		\draw (5.6,0) -- (3.8,0) node[anchor=east] {\footnotesize $0.1$};
        		\draw (5.6,1) -- (3.8,1) node[anchor=east] {\footnotesize $1$};
        		\draw (5.6,2) -- (3.8,2) node[anchor=east] {\footnotesize $10$};
        		\draw (5.6,3) -- (3.8,3) node[anchor=east] {\footnotesize $100$};
        		\draw (5.6,4) -- (3.8,4) node[anchor=east] {\footnotesize $1000$};
        		\draw (5.6,5) -- (3.8,5) node[anchor=east] {\footnotesize $10000$};

            \begin{scope}
                \clip (0,-1) rectangle (95,5.5);
        		\draw[color=blue] (5,-1.090016305) -- (6,-0.870277891) -- (7,-0.711907345) -- (8,-0.603086138) -- (9,-0.403380915) -- (10,-0.204057449) -- (11,0.026087711) -- (12,0.174051081) -- (13,0.343231295) -- (14,0.466691199) -- (15,0.619250943) -- (16,0.751066525) -- (17,0.878896184) -- (18,1.007334191) -- (19,1.111144176) -- (20,1.223237052) -- (21,1.324336584) -- (22,1.420857253) -- (23,1.520710381) -- (24,1.61678188) -- (25,1.708778132) -- (26,1.782194113) -- (27,1.87173419) -- (28,1.932203007) -- (29,2.00431888) -- (30,2.092237776) -- (31,2.183358715) -- (32,2.267076995) -- (33,2.343038232) -- (34,2.42430176) -- (35,2.48823437) -- (36,2.583749299) -- (37,2.645578892) -- (38,2.716312418) -- (39,2.789089301) -- (40,2.859601628) -- (41,2.925078128) -- (42,2.99044048) -- (43,3.053114242) -- (44,3.120105292) -- (45,3.173972533) -- (46,3.240970075) -- (47,3.297860044) -- (48,3.367709816) -- (49,3.433848963) -- (50,3.483554445) -- (51,3.546140417) -- (52,3.602278704) -- (53,3.6463168) -- (54,3.696901021) -- (55,3.742404578) -- (56,3.789930762) -- (57,3.842351077) -- (58,3.893754843) -- (59,3.944548413) -- (60,4.017877584) -- (61,4.053750162) -- (62,4.090935196) -- (63,4.153389925) -- (64,4.184594517) -- (65,4.228210559) -- (66,4.273470988) -- (67,4.319667125) -- (68,4.381914474) -- (69,4.424154531) -- (70,4.461932978) -- (71,4.506255809) -- (72,4.552459978) -- (73,4.600995089) -- (74,4.642668888) -- (75,4.690837655) -- (76,4.73494759) -- (77,4.785512049) -- (78,4.868979097) -- (79,4.886867743) -- (80,4.911986306) -- (81,4.973505031) -- (82,5.000213935) -- (83,5.039509419) -- (84,5.085127974) -- (85,5.136792812) -- (86,5.186392506) -- (87,5.22538154) -- (88,5.262721107) -- (89,5.286734813) -- (90,5.322216548) -- (91,5.368717712) -- (92,5.399755366) -- (93,5.440528671);

        		\draw[color=Maroon,dash pattern={on 7pt off 3pt}] (5,-0.666492271) -- (6,-0.454853763) -- (7,-0.410854979) -- (8,-0.291409155) -- (9,-0.159429735) -- (10,-0.036684489) -- (11,0.085001428) -- (12,0.235189969) -- (13,0.3591883) -- (14,0.498069212) -- (15,0.60637599) -- (16,0.712531854) -- (17,0.831039408) -- (18,0.924362001) -- (19,1.036843196) -- (20,1.130583336) -- (21,1.229860437) -- (22,1.321301965) -- (23,1.415424613) -- (24,1.498138594) -- (25,1.602162581) -- (26,1.669748543) -- (27,1.755742795) -- (28,1.835525255) -- (29,1.912445631) -- (30,1.967709028) -- (31,2.03199245) -- (32,2.099324671) -- (33,2.15995025) -- (34,2.204308236) -- (35,2.276828479) -- (36,2.340747143) -- (37,2.390110737) -- (38,2.460453373) -- (39,2.530335897) -- (40,2.56822838) -- (41,2.653309746) -- (42,2.718199006) -- (43,2.747042609) -- (44,2.826729297) -- (45,2.882826401) -- (46,2.944281369) -- (47,2.96600438) -- (48,3.03302582) -- (49,3.065141801) -- (50,3.108893627) -- (51,3.16827359) -- (52,3.219977194) -- (53,3.301699355) -- (54,3.31944362) -- (55,3.353045428) -- (56,3.422480653) -- (57,3.459108952) -- (58,3.490852721);

        		\draw[color=ForestGreen,dashed] (5,0.436850927) -- (6,0.536363977) -- (7,0.586433137) -- (8,0.574025477) -- (9,0.712591623) -- (10,0.714429519) -- (11,0.885855175) -- (12,1.393347205) -- (13,1.420151055) -- (14,1.75334417) -- (15,1.935128781) -- (16,2.293219756) -- (17,2.499231604) -- (18,2.599544922) -- (19,2.830517392) -- (20,3.051578948) -- (21,3.285806916) -- (22,3.776099431) -- (23,3.925444163);
    	 \end{scope}

        	\draw[thick,-to] (5,-1) --node[anchor=north,shift={(0,-0.75)}]{number of variables ($n$)} (93,-1);
        	\draw[thick,-to] (5,-1) --node[anchor=south,shift={(-1.1,0)},rotate=90]{time (seconds)} (5,5.5);
            
            \node[draw=black,fill=white,thick,anchor=south east,shift={(-0.1,0.1)}] at (92,-1) {%
                \begin{tabular}{@{}r@{ }l@{}}
                 \raisebox{2pt}{\tikz{\draw[blue] (0,0) -- (5mm,0);}}&ours\\
                 \raisebox{2pt}{\tikz{\draw[Maroon,dash pattern={on 7pt off 3pt}] (0,0) -- (5mm,0);}}&DSOS\\
                 \raisebox{2pt}{\tikz{\draw[ForestGreen,dashed] (0,0) -- (5mm,0);}}&RSOS
                \end{tabular}};
    	\end{tikzpicture}
            \caption{The time required to run level $k=d+1$ of our hierarchy, as well as the DSOS and RSOS hierarchies (lower is better).}\label{fig:DSOS_timings}
        \end{subfigure}

    	\caption{Plots illustrating how level $k=d+1$ of our HRSOS hierarchy of lower bounds (blue, solid) compares to level $k=d+1$ of the DSOS (red, long dashes) and RSOS (green, short dashes) hierarchies, when applied to random quartic polynomials ($d=2$). The RSOS and DSOS curves stop at $n = 23$ and $n = 58$ variables, respectively, due to memory limitations.}\label{fig:dsos_compare}
    \end{figure}
\end{example}

\section{Tensor optimization hierarchy}\label{sec:tensor}


Let $\V_i=\real^{n_i}$ for $i=1,\dots, m$, let $D_1,\dots, D_m$ be positive integers, let $\U=S^{D_1}(\V_1^*)\ootimes S^{D_m}(\V_m^*)$, and let $p \in \U$. In this section we present a hierarchy of eigencomputations for computing
\ba\label{eq:pminbd}
p_{\min}:= \min_{\substack{x_j \in \V_j\\ \norm{x_j}=1}} p(x_1,\dots, x_m).
\ea
This analysis is similar to that of Section~\ref{sec:converge}. We then apply this hierarchy to compute the real tensor spectral norm and to optimize real biquadratic forms.

The following proposition establishes that it suffices to assume $D_1,\dots, D_m$ are all even, at the expense of adding up to $m$ new variables.
\begin{prop}\label{prop:even_gen}
If $D_1$ is odd, then
\ba
p_{\min}&=\frac{(D_1+1)^{(D_1+1)/2}}{D_1^{D_1/2}} (f \cdot p)_{\min}\\
&=\frac{(D_1+1)^{(D_1+1)/2}}{D_1^{D_1/2}} \min_{\substack{x_j \in \V_j \\ \norm{x_j}=1\\ 1 \neq j \in [m]}} \;\;\;\min_{\substack{x_1 \in \V_1'\\ \norm{x_1}=1}}  (f \cdot p)(x_1,\dots, x_m),
\ea
where $\V_1'= \V_1\oplus \real^1$ and $f=(\vec{0},1) \in (\V_1')^*$.
\end{prop}

\begin{proof}[Proof of Proposition~\ref{prop:even_gen}]
For $v:=(v_2,\dots, v_m)$ fixed, define a polynomial $p_v \in S^{D_1}(\V_1^*)$ by
\ba
p_v(x_1)=p(x_1,v_2,\dots, v_m).
\ea
By~\cite[Lemma B.2]{doherty2012convergence}, we have
\ba
p_{\min}&=\min_{\substack{v_j \in \V_j \\ \norm{v_j}=1\\ 1 \neq j \in [m]}}\;\;\;\min_{\substack{x_1 \in \V_1 \\ \norm{x_1}=1}}   p_v(x)\\
&=\frac{(D_1+1)^{(D_1+1)/2}}{D_1^{D_1/2}} \min_{\substack{v_j \in \V_j \\ \norm{v_j}=1\\ 1 \neq j \in [m]}} \;\;\;\min_{\substack{x_1 \in \V_1' \\ \norm{x_1}=1}}   (f \cdot p_v)(x_1)\\
&=\frac{(D_1+1)^{(D_1+1)/2}}{D_1^{D_1/2}}  (f \cdot p)_{\min}.
\ea
This completes the proof.
\end{proof}

In the remainder of this section, we assume $D_1,\dots, D_m$ are even without loss of generality. Let $d_j=D_j/2$ for all $j \in [m]$, and let $\bd=(d_1,\dots, d_m)$.
\subsection{Reduction for tensor optimization}
In this section, we generalize the reduction proven in Theorem~\ref{thm:real_opt} to the setting of tensor optimization.
\begin{theorem}\label{thm:real_optbd}
Let $p \in S^{2d_1}(\V_1^*)\ootimes S^{2d_m}(\V_m^*)$, let $\V_i^{\complex}=\complex^{n_i}$ (the complexification of $\V_i$), and let $P=M(p)\in \Herm(S^d(\V_1^{\complex})\ootimes S^d(\V_m^{\complex}))$. Then
\ba\label{eq:real_optbd}
P_{\min} \leq p_{\min} \leq \frac{P_{\min}}{\delta(\bd)},
\ea
where $\delta(\bd)=\delta(d_1)\cdots \delta(d_m)$, and $\delta(d_i)$ is defined in~\eqref{eq:delta}.
\end{theorem}
\begin{proof}
Let $z_i \in \V_i^{\complex}$ be unit vectors for which
\ba
P_{\min}=P(\bfz,\bfz^\dg),
\ea
where $\bfz=(z_1,\dots, z_m)$. By definition,
\ba
P(\bfz,\bfz^\dg)= \tr((z_1z_1^\dg)^{\otimes d_1}\otimes \dots \otimes (z_m z_m^\dg)^{\otimes d_m} P).
\ea
Let $x_j,y_j \in \V$ be defined by $x_j + i y_j = z_j$ for each $j \in [m]$. Note that $P$ is invariant under partial transposition along any of the $d_1$ factors of $\V_1^{\complex}$. In particular, $P=\frac{1}{2}(P+P^{\t_1})$, where $(\cdot)^{\t_1}$ denotes the partial transpose on the first factor of $\V_1^{\complex}$. Thus,
\ba
&\tr((z_1z_1^\dg)^{\otimes d_1}\otimes \dots \otimes (z_mz_m^\dg)^{\otimes d_m} P)\\
&=\frac{1}{2} \tr((z_1z_1^\dg)^{\otimes d_1}\otimes \dots \otimes (z_m z_m^\dg)^{\otimes d_m} P)+ \frac{1}{2} \tr((z_1z_1^\dg)^{\otimes d_1}\otimes \dots \otimes (z_mz_m^\dg)^{\otimes d_m}P^{\t_1})\\
					&=\frac{1}{2} \tr((z_1z_1^\dg)^{\otimes d_1}\otimes \dots \otimes (z_mz_m^\dg)^{\otimes d_m}P)+ \frac{1}{2} \tr((z_1z_1^\dg)^{\t} \otimes (z_1z_1^\dg)^{\otimes d_1-1}\otimes \dots \otimes (z_mz_m^\dg)^{\otimes d_m}P)\\
					&=\tr\bigg[(x_1x_1^\t+y_1y_1^\t)\otimes (z_1z_1^\dg)^{\otimes d_1-1}\otimes \dots \otimes (z_mz_m^\dg)^{\otimes d_m} M(p)\bigg].
\ea
Continuing in this way for the other factors, we obtain
\ba
\tr((z_1z_1^\dg)&^{\otimes d_1}\otimes \dots \otimes (z_mz_m^\dg)^{\otimes d_m} \;\;P)\\
&=\tr((x_1x_1^\t+y_1y_1^\t)^{\otimes d_1}\otimes \dots \otimes (x_m x_m^\t + y_m y_m^\t)^{\otimes d_m} \;\;P)\\
&=\tr\bigg[M\bigg(((x_1^{\t})^2+(y_1^{\t})^2)^{d_1} \ootimes ((x_m^{\t})^2+(y_m^{\t})^2)^{d_m}\bigg) \;\;P\bigg]\\
&=\tr\bigg[M\bigg(((x_1^{\t})^{2}+(y_1^{\t})^{2})^{d_1}\bigg)\ootimes M\bigg(((x_m^{\t})^{2}+(y_m^{\t})^{2})^{d_m}\bigg) \;\;P\bigg].
\ea
For each $j \in [m]$, let
\ba
\delta^{(j)}=\tr\bigg[M\bigg(((x_j^{\t})^{2}+(y_j^{\t})^{2})^{d_j}\bigg)\bigg].
\ea
By~\cite[Corollary 5.6]{reznick2013length} for each $j\in [m]$ there exist real unit vectors ${v_{j,1}},\dots, {v_{j,d_j+1}}\in \V_j$ for which
\ba
\frac{1}{\delta^{(j)}} M\bigg(((x_j^{\t})^{2}+(y_j^{\t})^{2})^{d_j}\bigg) \in \setft{conv} \{(v_{j,1}v_{j,1}^\t)^{\otimes d_1},\dots, (v_{m,d_m+1}v_{m,d_m+1}^\t)^{\otimes d_m}\}.
\ea
Taking the tensor product over all $j \in [m]$, multiplying by $P$ and taking the trace, this gives
\ba
\frac{P_{\min}}{\delta^{(1)}\cdots \delta^{(m)}} \in \setft{conv} \{ \tr((v_{1,i_1}v_{1,i_1}^\t)^{\otimes d_1} \otimes \dots \otimes (v_{m,i_m}v_{m,i_m}^\t)^{\otimes d_m}\;\;P) : i_j \in [d_j+1]\}.
\ea
Thus, there exist $i_1 \in [d_1+1],\dots, i_m \in [d_m+1]$ for which
\ba
\tr((v_{1,i_1}v_{1,i_1}^\t)^{\otimes d_1} \otimes \dots \otimes (v_{m,i_m}v_{m,i_m}^\t)^{\otimes d_m}\;\;P) \leq \frac{P_{\min}}{\delta^{(1)}\cdots \delta^{(m)}}.
\ea
Note that $\delta^{(j)} \geq \delta({d_j})$ for each $j \in [m]$, by the proof of Theorem~\ref{thm:real_opt}. This completes the proof.
\end{proof}

\subsection{m-HRSOS hierarchy}
For $\bfd=(d_1,\dots,d_m)$, let
\ba
N^{(\bfd)}=M(\norm{x_1}^{2d_1} \ootimes \norm{x_m}^{2d_m})\in \Herm(S^{d_1}(\V_1^{\complex})\ootimes S^{d_m}(\V_m^{\complex})).
\ea
Recall that for a Hermitian operator $H \in \Herm(S^{d_1}(\V_1^{\complex})\ootimes S^{d_m}(\V_m^{\complex}))$ and an integer $k \geq d:=\max_j d_j$ we define
\ba
H_k(\bfz,\bfz^\dg)= H(\bfz,\bfz^\dg) \cdot (\norm{z_1}^{2(k-d_1)}\ootimes \norm{z_m}^{2(k-d_m)})\in \Herm(S^k(\V_1^{\complex})\ootimes S^k(\V_m^{\complex})).
\ea

\begin{theorem}\label{thm:convergencegen}
Let $p \in S^{2d_1}(\V_1^*)\ootimes S^{2d_m}(\V_m^*)$, and let $P=M(p)$. Furthermore, let
\ba
\eta_k=\lambda_{\min}((N_k^{(\bfd)})^{-1/2}P_k (N_k^{(\bfd)})^{-1/2}).
\ea
Then $\eta_d \leq \eta_{d+1} \leq \dots$ and $\lim_k \eta_k = \pmin$. Furthermore,
\ba
p_{\min}-\eta_k \leq \norm{P}_{\infty}\left(1+\kappa(N^{(\bfd)})\right) \frac{4 |\bfd| (\max_j n_j-1)}{\delta(\bfd)(k+1)} = O\left(\frac{1}{k}\right),
\ea
where $|\bfd|=d_1+\dots+d_m$.
\end{theorem}

Numerics indicate that
\ba
\kappa(N^{(\bfd)})=\binom{n_1/2+d_1-1}{\floor{d_1/2}}\cdots \binom{n_m/2+d_m-1}{\floor{d_m/2}},
\ea
where the binomial function is extended to non-integer inputs via the Gamma function. However we do not have a proof of this fact.

Now we can prove Theorem~\ref{thm:convergencegen}.

\begin{proof}[Proof of Theorem~\ref{thm:convergencegen}]
By a similar argument as in Theorem~\ref{thm:convergence}, we have
\ba
\eta_k= \max\{ \eta: P_k -\eta N^{(\bfd)}_k \succeq 0\}.
\ea
Let
\ba
q^{(k)}(\bfx) = p(\bfx)-\eta_k (\norm{x_1}^{2d_1}\ootimes \norm{x_m}^{2d_m}) \in S^{2d_1}(\V_1^*)\ootimes S^{2d_m}(\V_m^*),
\ea
and let $Q^{(k)}=M(q^{(k)}).$ For the inequality $\eta_k \leq p_{\min}$, note that since $Q^{(k)}_k \succeq 0$, we have
\ba
p(\bfx)-\eta_k&=Q_k^{(k)}(\bfx,\bfx^\t) \geq 0
\ea
for all $\bfx=(x_1,\dots, x_m)$ with $x_i \in \V_i$, so $p_{\min} - \eta_k \geq 0$. For the inequality $\eta_{k+1} \geq \eta_k$, let $\psi$ be a unit eigenvector for $P_{k+1}$ with eigenvalue $\eta_{k+1}$. Then
\ba
\eta_{k+1}=\tr(P_{k+1} \psi \psi^\dg) = \tr(P_k \tr_{1,\dots,1}(\psi \psi^\dg))\geq \eta_k,
\ea
where $\tr_{1,\dots, 1}$ denotes the partial trace over one copy of $\V_j^{\complex}$ for each $j \in [m]$, the second equality follows from $\tr_{1,\dots,1}(P_{k+1})=P_k$, and the inequality follows by convexity.

For the bound, note that
\ba\label{eq:qmingen}
Q_{\min}^{(k)} &\leq \norm{Q^{(k)}}_{\infty} \frac{4d (\max_j n_j-1)}{k+1}\\
		&\leq \norm{P}_{\infty} \left(1+\kappa(N^{(\bfd)})\right) \frac{4d(\max_j n_j-1)}{k+1}.
\ea
The first line follows from Theorem~\ref{thm:mhsos} and the fact that $\lambda_{\min}(Q_k^{(k)}) = 0$. The second line follows from
\ba
\norm{Q^{(k)}}_{\infty}&\leq \norm{P}_{\infty}+ \abs{\eta_k} \norm{N^{(\bfd)}}_{\infty} \\
			&\leq \norm{P}_{\infty}\left(1+\frac{\lambda_{\max}(N^{(\bfd)})}{\lambda_{\min}(N^{(\bfd)})}\right)
\ea
Here, the first line is the triangle inequality. The second line follows from $\norm{N^{(\bfd)}}_{\infty}=\lambda_{\max}(N^{(\bfd)})$ by Proposition~\ref{prop:Mspos}, and $\abs{\eta_k}\leq \norm{P}_{\infty} \;\lambda_{\min}(N^{(\bfd)})^{-1}$ since choosing $\eta$ equal to minus the righthand side would guarantee $M(p-\eta \; \norm{x}^{2d}) \succeq 0$.

It follows from Theorem~\ref{thm:real_optbd} and~\eqref{eq:qmingen} that
\ba
p_{\min}-\eta_k &= q_{\min}^{(k)}\\
 &\leq \frac{Q_{\min}^{(k)}}{\delta(\bfd)} \\
 &=\norm{P}_{\infty}\left(1+\kappa(N^{(\bfd)})\right) \frac{4d(\max_jn_j-1)}{\delta(\bfd) (k+1)}.
\ea
This completes the proof.
\end{proof}

\subsection{Examples: Biquadratic forms and the tensor spectral norm}\label{sec:examples}

In this section we use our tensor optimization hierarchy and Theorem~\ref{thm:convergencegen} to give hierarchies of eigencomputations for two tasks: minimizing a biquadratic form over the unit sphere, and computing the real spectral norm of a real tensor.
\begin{example}[Biquadratic forms]
Let $\V_1=\real^{n_1}, \V_2=\real^{n_2}$ and $\V_1^{\complex}=\complex^{n_1}, \V_2^{\complex}=\complex^{n_2}$ (the complexifications of $\V_1, \V_2$). A \textit{biquadratic form} is an element $p\in S^2(\V_1^*) \otimes S^2(\V_2^*)$. Let $P=M(p)$ be the maximally symmetric Gram operator of $p$ (see Section~\ref{background}), and let
\ba
P_k(\bfz,\bfz^\dg)=P(\bfz,\bfz^\dg) \cdot (\norm{z_1}^{2(k-2)} \otimes \norm{z_2}^{2(k-2)})\in \Herm(S^2(\V_1^{\complex}) \otimes S^2(\V_2^{\complex}))
\ea
for $k \geq 2$ and $\bfz=(z_1,z_2)$ with $z_i \in \V_i^{\complex}$. Since $P_k$ is real-valued, we can regard it as the real symmetric operator
\ba
P_k=(\Pi_{\V_1,k} \otimes \Pi_{\V_2, k}) (P \otimes \I_{\V_1}^{\otimes k-2}\otimes \I_{\V_2}^{\otimes k-2})(\Pi_{\V_1,k} \otimes \Pi_{\V_2,k}) \in \Sym(S^k(\V_1)\otimes S^k(\V_2)),
\ea
where $\Pi_{\V_i,k}$ is the projection onto $S^k(\V_i)$. Let $\eta_k:=\lambda_{\min}(P_k)$. By Theorem~\ref{thm:convergencegen}, the $\eta_k$ converge to $p_{\min}$ from below at a rate of $O(1/k)$.

In particular, suppose we wish to determine whether $p(x,y)>0$ for all $x,y \neq 0$. Our hierarchy shows that this holds if and only if $P_k$ is positive definite for some $k$.

\end{example}

\begin{example}[Tensor spectral norm]
Let $\V_i=\real^{n_i}$ for $i=1,\dots, m$ and $\V_i^{\complex}=\complex^{n_i}$ (the complexification of $\V_i$), and let $p \in \V_1 \ootimes \V_m$. Assume without loss of generality that $\norm{p}_2=1$. The \textit{(real) spectral norm} of $p$ is defined as
\ba
\norm{p}_{\sigma, \real} := \max_{\substack{v_j \in \V_j \\ \norm{v_j}=1}} \abs{\ip{p}{v_1\otimes \dots \otimes v_m}}.
\ea
We can use our hierarchy to compute the spectral norm of $p$ as follows. First note that
\ba
- \norm{p}_{\sigma, \real} = \min_{\substack{v_j \in \V_j \\ \norm{v_j}=1}} \ip{p}{v_1 \otimes \dots \otimes v_m},
\ea
so it suffices to compute the righthand side of this expression. For each $j$ let $\U_j=\V_j \oplus \real^1$, let $u_j=(\vec{0},1) \in \U_j$, and let
\ba
r = p \cdot (u_1 \ootimes u_m) \in S^2(\U_1)\ootimes S^2(\U_m).
\ea
This product is defined in Section~\ref{background}, and we give more details here by way of example. Let
\ba
q= p \otimes u_1 \ootimes u_m \in \bigg(\bigotimes_{j=1}^m \U_j \bigg)^{\otimes 2}.
\ea
Then
\ba
r= (\Pi_{\U_1,k}\otimes \dots \otimes \Pi_{\U_m,k}) (\sigma \cdot q) \in S^2(\U_1)\ootimes S^2(\U_m),
\ea
where $\sigma\in \frakS_{2m}$ is the permutation that sends $(1,2,\dots, 2m)$ to $(1,m+1,2,m+2,\dots,m,2m)$. Direct calculation shows that $\norm{r}_2 = 2^{-m/2}$. By Proposition~\ref{prop:even_gen} it holds that
\ba\label{eq:tensor}
2^{-m} \;\cdot \;\min_{\substack{v_j \in \V_j \\ \norm{v_j}=1}} \ip{p}{v_1 \otimes \dots \otimes v_m}=  \min_{\substack{v_j \in \U_j \\ \norm{v_j}=1}} \ip{r}{v_1^{\otimes 2} \otimes \dots \otimes v_m^{\otimes 2}}.
\ea

Our hierarchy can be used to compute the righthand side of~\eqref{eq:tensor}, as follows. Let
\ba
R=M(r^{\t}) \in \Herm(\U_1^{\complex} \ootimes \U_m^{\complex})
\ea
be the maximally symmetric Gram operator of $r^{\t}$ (see Section~\ref{background}). 
For each $k \in \natural$, let
\ba
R_k(\bfz ,\bfz^\dg)= R(\bfz,\bfz^\dg) \cdot (\norm{z_1}^{2(k-1)} \ootimes \norm{z_m}^{2(k-1)}) \in \Herm(S^k(\U_1^{\complex}) \ootimes S^k(\U_m^{\complex})),
\ea
where $\bfz=(z_1,\dots, z_m)$ and $z_i \in \U_i$. Since $R_k$ is real-valued, we can regard it as the real symmetric operator
\ba
R_k=(\Pi_{\V_1,k}\otimes \dots \otimes \Pi_{\V_m,k})(R \otimes \I_{\V_1}^{\otimes k-1}\otimes  \dots \otimes \I_{\V_m}^{\otimes k-1})(\Pi_{\V_1,k}\otimes &\dots \otimes \Pi_{\V_m,k})\\
&\in \Sym(S^k(\U_1) \ootimes S^k(\U_m)).
\ea
Then by our Theorem~\ref{thm:convergencegen}, the minimum eigenvalues $\eta_k := \lambda_{\min}(R_k)$ converge to the righthand side of~\eqref{eq:tensor} from below at a rate of $O(1/k)$. More precisely, we have the following convergence guarantee:

\begin{theorem}[Hierarchy for tensor spectral norm]
For each $k \in \natural$ let $\mu_k=-2^m \lambda_{\min}(R_k)$. Then $\mu_1 \geq \mu_2 \geq \dots$ and $\lim_k \mu_k = \norm{p}_{\sigma, \real}$. Moreover,
\ba\label{eq:spectral}
\mu_k - \norm{p}_{\sigma, \real} 
&\leq \frac{2^{m/2+3} m (\max_j n_j-1)}{k+1}= O\left(\frac{1}{k}\right).
\ea
\end{theorem}
\begin{proof}
Note that
\ba
\mu_k - \norm{p}_{\sigma, \real} &= -2^m \eta_k + \min_{\substack{v_j \in \V_j\\ \norm{v_j}=1}} \ip{p}{v_1 \otimes \dots \otimes v_m}\\
&=-2^m \eta_k + 2^m \min_{\substack{v_j \in \V_j\\ \norm{v_j}=1}} \ip{r}{v_1^{\otimes 2} \otimes \dots \otimes v_m^{\otimes 2}}\\
&= 2^{m}(r^{\t}_{\min}-\eta_k).
\ea
The bound follows from Theorem~\ref{thm:convergencegen}, $\delta(1,\dots, 1)=1$,
\ba
{\kappa(N^{(1,\dots,1)})=\kappa(\I_{\U_1}\ootimes \I_{\U_m})=1},
\ea
and $\norm{M(r^{\t})}_{\infty} \leq \norm{r}_2 = 2^{-m/2}$. This completes the proof.
\end{proof}

%
%
%
%
%
\end{example}

\section{Constrained polynomial optimization}\label{constraints}

Let $\V_{\real}=\real^n$, let $p \in S^{2d}(\Vr^*)$, and let $q_i\in S^{c_i}(\Vr^*)$, $i=1,\dots, \ell$ be forms. We consider the constrained polynomial optimization problem

\begin{align}\label{eq:constrained_optimization_problem}
p_{q,\min}=\min_{\substack{x \in \Vr \\ \norm{x}=1 \\ q_1(x)=\dots = q_{\ell}(x)=0}} p(x).
\end{align}
Analogously, for $\Vc=\complex^n$, a Hermitian form $H \in \Herm(S^d(\Vc))$, and forms $q_i \in S^{c_i}(\Vc^*)$, let
\ba
H_{q,\min}=\min_{\substack{z \in \Vc \\ \norm{z}=1 \\ q_1(z)=\dots = q_{\ell}(z)=0}} H(z,z^\dg).
\ea

For a finite dimensional Hilbert space $\V$ over $\field \in \{\real, \complex\}$, let ${S^{\bullet}(\V^*)=\bigoplus_{d=0}^{\infty} S^d(\V^*)}$ be the symmetric algebra. For forms $q_i \in S^{c_i}(\V^*)$, $i=1,\dots, \ell$, let
\ba
V_{\field}(q_1,\dots, q_\ell)=\{v : q_1(v)=\cdots = q_{\ell}(v)=0\}\subseteq \V
\ea
be the \textit{variety} they cut out, and let
\ba
I=\langle q_1,\dots, q_{\ell} \rangle:=\bigg\{\sum_{i=1}^{\ell} q_i f_i : f_i \in S^{\bullet}(\V^*) \bigg\} \subseteq S^{\bull}(\V^*)
\ea
be the ideal they generate. For each positive integer $k$, let $I_k \subseteq S^k(\V^*)$ be the degree-$k$ part of $I$, let $I_k^{\perp} \subseteq S^k(\V)$ be the orthogonal complement to $I$ with respect to the bilinear pairing $(\cdot,\cdot): \V^* \times \V \rightarrow \field$, and let $\Pi_{I,k} \in \End(S^k(\V))$ be the orthogonal projection onto $I_k^{\perp}$.

The following hierarchy of eigencomputations for $H_{q,\min}$ was recently obtained in~\cite[Theorem 7.2]{DJLV24}. See also~\cite{catlin1999isometric,d2009polynomial}.

\begin{theorem}\label{thm:chsos}
Let $H \in \Herm(S^d(\Vc))$ be a Hermitian form, let $q_i \in S^{c_i}(\Vc^*)$ be forms for $i=1,\dots,\ell$, and let ${I=\langle q_1,\dots, q_{\ell} \rangle}.$ For each $k \geq d$, let $H_k = H(z,z^\dg) \norm{z}^{2(k-d)}$ and let $\mu_k=\lambda_{\min}(\Pi_{I,k} H_k \Pi_{I,k})$. Then $\mu_d \leq \mu_{d+1} \leq \dots$ and $\lim_k \mu_k = H_{q, \min}$.
\end{theorem}

In the unconstrained setting, we obtained our hierarchy of eigencomputations for $p_{\min}$ by relating it to $P_{\min}$ up to a positive constant, where $P=M(p) \in \Herm(S^d(\Vc))$ is the maximally symmetric Gram operator of $p$. In more details, we proved that $P_{\min} \leq p_{\min} \leq \frac{P_{\min}}{\delta(d)}$, where $\delta(d)$ is defined in~\eqref{eq:delta}; in particular,
\ba
\pmin=0 \iff \Pmin = 0.
\ea
It is natural to ask: Can we do the same in the constrained setting? If ($p_{q,\min}= 0 \iff P_{q,\min}=0$), then we can again get a hierarchy of eigencomputations for $p_{q,\min}$ using Theorem~\ref{thm:chsos} and similar arguments as in Theorem~\ref{thm:convergence}. Note that $P_{q,\min} \leq p_{q,\min}$ always holds, since $P(x,x^\t)=p(x)$ for any $x \in \Vr$. The following proposition shows that the equivalence ($p_{q,\min}= 0 \iff P_{q,\min}=0$) does not always hold.





\begin{prop}\label{prop:constraints_counterexample}
Let $\V=\real^3$, and let
\ba\label{eq:example}
p(x,y,z)&=xz \in S^2(\V^*),\\
q(x,y,z)&=x^2+y^2-\frac{1}{\sqrt{3}} x z \in S^2(\V^*).
\ea
Then
\ba\label{eq:counterexample}
p_{q, \min} =0 > -\frac{1}{2\sqrt{3}} \geq   P_{q,\min}.
\ea
\end{prop}

The example provided by this proposition is non-trivial in two ways:
\begin{enumerate}
\item Note that $p_{q,\min} \geq P_{r,\min}$ for any collection of complex forms $r_1,\dots, r_t$ for which $V_{\complex}(r_1,\dots, r_t) \cap \Vr = V_{\real}(q_1,\dots, q_{\ell})$. In the above example, $V_{\complex}(q)=\cl_{\complex}(V_{\real}(q))$, where $\cl_{\complex}(\cdot)$ denotes the complex Zariski closure~\cite{harris2013algebraic}, so there is no such collection of forms $r_1,\dots, r_t$ for which ($p_{q,\min}=0 \iff P_{r,\min}=0$) holds. Indeed, it is easily checked that $q$ is irreducible over $\complex$ and that $V_{\complex}(q)$ contains at least one real smooth point (the only singular point of $V_{\complex}(q)$ is $(0,0,0)$). Hence $\setft{Cl}_{\complex}(V_{\real}(q))=V_{\complex}(q)$ by~\cite[Theorem 2.2.9(2)]{mangolte2020real}.
\item It is also natural to ask if equality holds under a different choice of Hermitian Gram operator for $p$. This is not the case because when $p$ is quadratic, $P$ is the unique Hermitian Gram operator for $p$.
\end{enumerate}

\begin{proof}[Proof of Proposition~\ref{prop:constraints_counterexample}]
Let us first verify that $p_{q,\min}=0$. It is clear that $p_{q,\min}\geq0$ because any solution must satisfy $xz=\sqrt{3}(x^2+y^2) \geq 0$. For the equality to zero, note that $q(0,0,1)=p(0,0,1)=0$. For the last inequality, consider the unit vector $v=(\frac{1}{\sqrt{6}},\frac{i}{\sqrt{3}},-\frac{1}{\sqrt{2}}) \in \complex^3$, which is a zero of $q$. Note that
\ba
M(p)=
\begin{bmatrix}
0 & 0 & 1/2 \\
0 & 0 & 0\\
1/2 & 0 & 0
\end{bmatrix},
\ea
and $P(v,v^\dg)=-\frac{1}{2\sqrt{3}}$. This completes the proof.
\end{proof}

We conclude this section by noting that there are two other natural ways to adapt our hierarchy to handle constraints, neither of which produce hierarchies of eigencomputations: First, one can use the method described in~\cite{ahmadi2019construction} for adapting certificates of global positivity to perform constrained polynomial optimization, but each step of the resulting hierarchy would require the use of bisection. Second, one could simply replace the sum-of-squares polynomials that appear in standard positivstellensatze with polynomials that our hierarchy certifies to be globally non-negative at some level. For example, a special case of Putinar's positivestellensatz says that a polynomial $p(x)$ is strictly positive on $V_{\real}(q)$ if and only if there exists a sum-of-squares polynomial $r(x)$ and a polynomial $h(x)$ for which $p(x)=r(x)+h(x) q(x)$~\cite{putinar1993positive,lasserre2001global}. Instead of searching for a sum-of-squares polynomial $r(x)$, one could impose that some fixed level of our hierarchy certifies global non-negativity of $r(x)$. This would give a sequence of semidefinite programs, indexed by the degrees of $r$ and $h$, and by the level of our hierarchy that we consider, which successfully finds such a pair $(r, h)$ in some high enough degrees and level if and only if $p(x)$ is strictly positive on $V_{\real}(q)$.

\section{Acknowledgments}
We thank Aravindan Vijayaraghavan for valuable insights in the early stages of this work. We thank an anonymous referee for insightful feedback to improve the exposition of this work. We thank another anonymous referee for pointing out the simpler lower bound on $\delta$ appearing at the end of the proof of Theorem~\ref{thm:real_opt}. We also thank Alexander Taveira Blomenhofer, Omar Fawzi, Monique Laurent, Victor Magron, Tim Netzer, Bruce Reznick, Jurij Volcic and Timo de Wolff for enlightening discussions. N.J. was supported by NSERC Discovery Grant RGPIN2022-04098. B.L. acknowledges that this material is based
 upon work supported by the National Science Foundation under Award No. DMS-2202782.

\appendix

\section{$M(p)_k$ in coordinates}\label{ap:mk}

Let $\Vr=\real^n$. In this section, for a real form $p \in S^{2d}(\Vr^*)$, we write $M(p)$ and $M(p)_k$ in coordinates. For a multi index $\bfalpha \in \integer_{\geq 0}^n$, let $|\bfalpha|=\alpha_1+\dots + \alpha_n$, and let $\binom{|\bfalpha|}{\bfalpha}=\frac{|\bfalpha|!}{\alpha_1! \cdots \alpha_n!}$ be the multinomial coefficient. Fix an orthonormal basis $e_1,\dots, e_n$ of $\Vr$, and consider the basis $\{e^{\bfalpha} : \bfalpha \in \integer_{\geq 0}^n, |\bfalpha|=d\} \subseteq S^d(\Vr)$, where\footnote{We note that the monomial basis is orthogonal but not normalized (with respect to the induced inner product from $\Vr$). An orthonormal basis is given by $\{\binom{d}{\bfalpha}^{1/2} e^{\bfalpha}\}$.
}
\ba
e^{\bfalpha}:=e_1^{\alpha_1} \cdots e_n^{\alpha_n} = \Pi_d (e_1^{\otimes \alpha_1} \ootimes e_n^{\otimes \alpha_n}) \in S^d(\Vr),
\ea
and $\Pi_d$ is the projection onto $S^d(\Vr)$. We also let $x^{\bfalpha}=x_1^{\alpha_1} \cdots x_n^{\alpha_n} \in \real$ when $x= x_1 e_1+\dots + x_n e_n \in \Vr.$

\begin{prop}
Let $p \in S^{2d}(\Vr^*)$ be a form, let $P=M(p) \in \Sym(S^d(\Vr))$ be the maximally symmetric Gram operator of $p$, and for $k \geq d$ let $P_k = M(p)_k \in \Sym(S^k(\Vr))$ be defined by $P_k= \Pi_k (P \otimes \I_{\Vr}^{\otimes k-d})\Pi_k$. If
\ba
p(x)=\sum_{\substack{\bfgamma \in \integer_{\geq 0}^n\\ |\bfgamma|=2d}} C_{\bfgamma} \binom{2d}{\bfgamma} x^{\bfgamma},
\ea
then
\ba\label{eq:M}
P=\sum_{|\bfalpha|=|\bfbeta|=d} C_{\bfalpha+\bfbeta} \binom{d}{\bfalpha} \binom{d}{\bfbeta} e^{\bfalpha}   (e^{\t})^{\bfbeta},
\ea
and
\ba\label{eq:Mk}
P_k=\sum_{\substack{|\bfalpha|=|\bfbeta|=d \\ |\bfgamma|=k-d}} C_{\bfalpha+\bfbeta} \binom{d}{\bfalpha} \binom{d}{\bfbeta} \binom{k-d}{\bfgamma} e^{\bfalpha+ \bfgamma}  (e^{\t})^{\bfbeta+\bfgamma}.
\ea
\end{prop}

\begin{proof}
Let $M \in \End(S^d(\Vr))$ be given by the righthand side of~\eqref{eq:M}. We use Proposition~\ref{prop:sym} to verify that $M=P$. First note that
\ba
\ip{x^{\otimes d}}{M x^{\otimes d}}&=\sum_{|\bfalpha|=|\bfbeta|=d} C_{\bfalpha+\bfbeta} \binom{d}{\bfalpha} \binom{d}{\bfbeta} x_1^{\alpha_1+\beta_1} \cdots x_n^{\alpha_n+\beta_n}\\
&=\sum_{\substack{\bfgamma \in \integer_{\geq 0}^n\\ |\bfgamma|=2d}} \sum_{\substack{|\bfalpha|=|\bfbeta|=d\\ \bfalpha+\bfbeta=\bfgamma}} C_{\bfgamma} \binom{d}{\bfalpha} \binom{d}{\bfbeta}  x_1^{\gamma_1} \cdots x_n^{\gamma_n}\\
&=\sum_{\substack{\bfgamma \in \integer_{\geq 0}^n\\ |\bfgamma|=2d}} C_{\bfgamma} \binom{2d}{\bfgamma} x_1^{\gamma_1} \cdots x_n^{\gamma_n}\\
&=p(x),
\ea
where the first two equalities are obvious, and the third follows from the identity
\ba
\sum_{\substack{|\bfalpha|=|\bfbeta|=d\\ \bfalpha+\bfbeta=\bfgamma}} \binom{d}{\bfalpha} \binom{d}{\bfbeta}=\binom{2d}{\bfgamma}.
\ea
It remains to prove that $M^{\t_1}=M$. Let $\bff_i \in \integer_{\geq 0}^n$ be the vector with $1$ in the $i$-th position and zeroes elsewhere, and note that $e^{\bfalpha}=\frac{1}{d} \sum_{i : \alpha_i >0} \alpha_i e_i \otimes e^{\bfalpha-\bff_i}$ when $|\bfalpha|=d$. Thus,
\ba
M&=\sum_{|\bfalpha|=|\bfbeta|=d} C_{\bfalpha+\bfbeta} \binom{d}{\bfalpha} \binom{d}{\bfbeta} e^{\bfalpha}   (e^{\t})^{\bfbeta}\\
&=\frac{1}{d^2}  \sum_{|\bfalpha|=|\bfbeta|=d} C_{\bfalpha+\bfbeta} \binom{d}{\bfalpha} \binom{d}{\bfbeta} \sum_{\substack{i,j \in [n]\\ \alpha_i, \beta_j >0}}  \alpha_i \beta_j (e_i \otimes e^{\bfalpha - \bff_i})  (e_j \otimes e^{\bfbeta - \bff_j})^\t.
\ea
Hence,
\ba
M^{\t_1}&=\frac{1}{d^2} \sum_{|\bfalpha|=|\bfbeta|=d} C_{\bfalpha+\bfbeta} \binom{d}{\bfalpha} \binom{d}{\bfbeta} \sum_{\substack{i,j \in [n]\\ \alpha_i, \beta_j >0}} \alpha_i \beta_j (e_j \otimes e^{\bfalpha - \bff_i})  (e_i \otimes e^{\bfbeta - \bff_j})^\t\\
&=\frac{1}{d^2} \sum_{i,j=1}^n  \sum_{\substack{|\bfalpha|=|\bfbeta|=d\\ \alpha_i, \beta_j>0}}C_{\bfalpha+\bfbeta}\binom{d}{\bfalpha} \binom{d}{\bfbeta}\alpha_i \beta_j (e_j \otimes e^{\bfalpha - \bff_i})  (e_i \otimes x^{\bfbeta - \bff_j})^\t\\
&= \frac{1}{d^2} \sum_{i,j=1}^n  \sum_{\substack{|\bfalpha'|=|\bfbeta'|=d\\ \alpha'_j, \beta'_i >0}}  C_{\bfalpha'+\bfbeta'} \binom{d}{\bfalpha'} \binom{d}{\bfbeta'} \alpha'_j \beta'_i (e_j \otimes e^{\bfalpha'-\bff_j})  (e_i \otimes e^{\bfbeta'-\bff_i})^\t\\
&=M,
\ea
where the third line follows from setting $\bfalpha'= \bfalpha - \bff_i + \bff_j$ and $\bfbeta'=\bfbeta-\bff_j+\bff_i$, and the rest are straightforward. Hence $M=P$.

Let $M_k$ be the righthand side of~\eqref{eq:Mk}. Let $\Vc=\complex^n$ (the complexification of $\Vr$), and regard $P$ as an element of $\Herm(S^d(\Vc))$ and $M_k$ as an element of $\Herm(S^k(\Vc))$. Recall that, as a Hermitian form, $P_k$ is given by
\ba
P_k(z,z^\dg)=P(z,z^\dg) \cdot \norm{z}^{2(k-d)}.
\ea
To complete the proof, it suffices to verify that $M_k(z,z^\dg)=P_k(z,z^\dg)$. Let
\ba
z = z_1 e_1 +\dots + z_n e_n \in \Vc.
\ea
Then
\ba
P_k (z,z^\dg) &=\sum_{|\bfalpha|=|\bfbeta|=d} C_{\bfalpha+\bfbeta} \binom{d}{\bfalpha} \binom{d}{\bfbeta} z^{\bfalpha}  (z^{\dg})^{\bfbeta} (z_1 z_1^\dg + \dots + z_n z_n^\dg)^{k-d}\\
&= \sum_{|\bfalpha|=|\bfbeta|=d} C_{\bfalpha+\bfbeta} \binom{d}{\bfalpha} \binom{d}{\bfbeta} z^{\bfalpha}  (z^{\dg})^{\bfbeta} \sum_{|\bfgamma|=k-d} \binom{k-d}{\bfgamma} z^{\bfgamma} (z^{\dg})^{\bfgamma} \\
&=\sum_{\substack{|\bfalpha|=|\bfbeta|=d \\ |\bfgamma|=k-d}} C_{\bfalpha+\bfbeta} \binom{d}{\bfalpha} \binom{d}{\bfbeta} \binom{k-d}{\bfgamma} z^{\bfalpha+ \bfgamma}  (z^{\dg})^{\bfbeta+\bfgamma}\\
&=\ip{z^{\otimes d}}{M_k z^{\otimes d}}\\
&=M_k(z,z^\dg).
\ea
This completes the proof.
\end{proof}

\begin{remark}\label{rmk:trace}
If $p(x)=x_1^{2j} x_2^{2(d-j)}$, then
\ba
\tr(M(p))=\binom{2d}{2j}^{-1} \binom{d}{j}^2 \tr(e^{(j,d-j)} (e^{\t})^{(j,d-j)})=\binom{2d}{2j}^{-1}\binom{d}{j}.
\ea
We use this formula in the proof of Theorem~\ref{thm:real_opt}.
\end{remark}

\section{Proof of Theorem~\ref{thm:mhsos}}\label{app:mhsos_proof}

Here we prove Theorem~\ref{thm:mhsos}. We first restate the theorem for convenience.

\begin{theorem}[m-HSOS hierarchy]\label{thm:app_mhsos}
Let $H \in \Herm(\U)$, let
\ba
H_k(\bfz,\bfz^\dg)=H(\bfz,\bfz^\dg) \cdot (\norm{z_1}^{2(k-d_1)}\ootimes \norm{z_n}^{2(k-d_m)})\in \Herm(S^k(\V_1)\ootimes S^k(\V_m))
\ea
when $k \geq \max_i d_i$, and let $\mu_k=\lambda_{\min}(H_k)$. Then $\mu_d \leq \mu_{d+1} \leq \dots $ and $\lim_k \mu_k = H_{\min}$. Furthermore,
\ba
H_{\min} - \mu_k \leq \norm{H}_{\infty} \frac{4 |\bfd| (\max_j n_j-1)}{k+1} = O(1/k),
\ea
where $|\bfd|=d_1+\dots + d_m$.
\end{theorem}

We prove this using a slightly more general quantum de Finetti theorem.
\begin{theorem}[More general quantum de Finetti theorem]\label{thm:definettigen}
Let $k \geq \max_j d_j$ be an integer, and let ${\psi \in S^k(\V_1)\ootimes S^k(\V_m)}$ be a unit vector. Then there exists a operator
\ba\label{eq:mtau}
\tau \in \setft{conv}\{(z_1z_1^\dg)^{\otimes d_1}\otimes \dots \otimes (z_m z_m^\dg)^{\otimes d_m} : z_j \in \V_j, \norm{z_j}=1\}\subseteq \Herm(\U)
\ea
for which
\ba
\norm{\tr_{[d_1+1 \cdot \cdot k], \dots , [d_m+1 \cdot \cdot k]}(\psi \psi^\dg)-\tau}_1 \leq \frac{4 |\bfd| (\max_j n_j-1)}{k+1},
\ea
where $|\bfd|=d_1+\dots+d_m$, and $\tr_{[d_1+1 \cdot \cdot k], \dots , [d_m+1 \cdot \cdot k]}(\psi\psi^\dg)$ denotes the partial trace over $k-d_j$ copies of $\V_j$ for each $j \in [m]$.
\end{theorem}
This theorem can be seen as a special case of~\cite[Theorem III.3, Remark III.4]{koenig2009most}, or a slight generalization of~\cite[Theorem II.2']{christandl2007one} and~\cite[Theorem 7.26]{Wat18}. We will prove the theorem as a corollary to~\cite[Theorem III.3, Remark III.4]{koenig2009most}.\begin{proof}[Proof of Theorem~\ref{thm:definettigen}]
Let
\ba
\A&=S^{d_1}(\V_1) \otimes \dots \otimes S^{d_m}(\V_m)\\
\B&=S^{k-d_1}(\V_1) \otimes \dots \otimes S^{k-d_m}(\V_m)\\
\C&=S^{k}(\V_1) \otimes \dots \otimes S^{k}(\V_m)\\
\X&=z_1^{\otimes d_1} \otimes \dots \otimes z_m^{\otimes d_m},
\ea
where the $z_j \in \V_j$ are arbitrary but fixed unit vectors. Then $\C \subseteq \A \otimes \B$ is an irreducible subrepresentation (with multiplicity one) of the product unitary group $\Unitary(\V_1) \times \dots \times U(\V_m)$. It is also straightforward to check that the quantity $\delta(\X)$ defined in~\cite[Definition III.2]{koenig2009most} is equal to
\ba
\frac{\dim(\B)}{\dim(\C)}&=\prod_{j=1}^m\frac{\binom{n_j+k-d_j-1}{k-d_j}}{\binom{n_j+k-1}{k}}\\
&\geq \prod_{j=1}^m \left(1-\frac{d_j(n_j-1)}{k+1}\right)\\
&\geq 1-\frac{d(\max_jn_j-1)}{k+1},
\ea
where the second line is a standard inequality that can be found e.g. in~\cite[Eq. (7.196)]{Wat18}.
The desired bound then follows directly from the bound given in~\cite[Theorem III.3, Remark III.4]{koenig2009most} in terms of $\delta(\X)$.
\end{proof}

\begin{proof}[Proof of Theorem~\ref{thm:app_mhsos}]
Let $\psi \in S^k(\V_1)\ootimes S^k(\V_m)$ be a unit eigenvector for $H_k$ with minimum eigenvalue, and let $\tau \in \Herm(\U)$ be an operator of the form~\eqref{eq:mtau} for which
\ba
\norm{\tr_{[d_1+1 \cdot \cdot k], \dots , [d_m+1 \cdot \cdot k]}(\psi \psi^\dg)-\tau}_1 \leq \frac{4 |\bfd| (\max_j n_j-1)}{k+1},
\ea
Then
\ba
H_{\min}- \mu_k &\leq \tr(H \tau) - \mu_k\\
&= \tr(H(\tau - \tr_{[d_1+1 \cdot \cdot k], \dots , [d_m+1 \cdot \cdot k]}(\psi \psi^\dg)))\\
 &\leq \norm{H}_{\infty} \norm{\tr_{[d_1+1 \cdot \cdot k], \dots , [d_m+1 \cdot \cdot k]}(\psi \psi^\dg)-\tau}_1\\
&\leq \norm{H}_{\infty} \frac{4 |\bfd| (\max_j n_j-1)}{k+1},
\ea
where the first line follows from the form of $\tau$ and convexity, the second line is straightforward, the third line follows from the inequality $\tr(AB)\leq \norm{A}_{\infty} \norm{B}_1$, and the fourth line follows from the quantum de Finetti theorem. This completes the proof.
\end{proof}

\section{Proof of Lemma~\ref{lemma:technical}}\label{ap:lemma}

In this section we prove Lemma~\ref{lemma:technical}, which was a technical ingredient for Theorem~\ref{thm:real_opt}.

\begin{lemma}[Lemma~\ref{lemma:technical} restated]
Let $d$ be a positive integer, and let
\ba
\delta=\sum_{j=0}^d \binom{d}{j}^2 \binom{2d}{2j}^{-1}\;\; t^j (1-t)^{d-j}.
\ea
Then
\ba
\min_{t \in [0,1]} \delta= {2^d}{\binom{2d}{d}^{-1}}.
\ea
\end{lemma}
\begin{proof}
If we make the change of variables $x = t-1/2$ then we see that
\ba
\delta & = \sum_{j=0}^d \frac{\binom{d}{j}^2}{\binom{2d}{2j}} \;\; \left(\frac{1}{2}+x\right)^j \left(\frac{1}{2}-x\right)^{d-j} \\
& = \sum_{j=0}^d \frac{\binom{d}{j}^2}{\binom{2d}{2j}} \;\; \left(\sum_{k=0}^j \binom{j}{k}\frac{x^k}{2^{j-k}}\right)\left(\sum_{\ell=0}^{d-j} \binom{d-j}{\ell}\frac{(-x)^\ell}{2^{d-j-\ell}}\right) \\
& = \sum_{j=0}^d \frac{\binom{d}{j}^2}{\binom{2d}{2j}} \;\; \sum_{k=0}^j \sum_{\ell=0}^{d-j} (-1)^\ell \binom{j}{k}\binom{d-j}{\ell}\frac{x^{k+\ell}}{2^{d-k-\ell}}.
\ea
If we define $s = k+\ell$ then the above expression can be rewritten as
\ba\label{eq:delta_formula}
\delta & = \frac{1}{2^d}\sum_{s=0}^d (-2)^{s} \sum_{k=0}^{s} (-1)^{k}\sum_{j=0}^d\frac{\binom{d}{j}^2\binom{j}{k}\binom{d-j}{s-k}}{\binom{2d}{2j}}x^s.
\ea
The coefficient of $x^s$ above equals
\ba\begin{split}\label{eq:xs_coeff_first}
& \hphantom{{} = {}} \frac{(-1)^s}{2^{d-s}} \sum_{k=0}^{s} (-1)^{k}\left(\sum_{j=0}^d\frac{\binom{d}{j}^2\binom{j}{k}\binom{d-j}{s-k}}{\binom{2d}{2j}}\right) \\
& = \frac{(-1)^s}{\binom{2d}{d}2^{d-s}} \sum_{k=0}^{s} (-1)^{k}\left(\sum_{j=0}^d \binom{2j}{j}\binom{j}{k}\binom{2d-2j}{d-j}\binom{d-j}{s-k}\right).
\end{split}\ea
To complete the proof, we first evaluate this coefficient explicitly. When $s=0$ then it is well-known that $\sum_{j=0}^d \binom{2j}{j} \binom{2d-2j}{d-j}=4^d$ (see e.g. \cite[Eq. (5.39)]{graham1994concrete}). We generalize this to arbitrary $s$ using generating functions, for which we prove a series of claims.

\textbf{Claim 1.} For each integer $k \geq 0$, the Taylor series (centered at $0$) of the function
\[
    f_k(y) = \frac{\binom{2k}{k}y^k}{(1-4y)^{k + 1/2}}
\]
is $f_k(y) = \sum_{d=0}^\infty \binom{2d}{d}\binom{d}{k} y^d$.

\emph{Proof of Claim~1.} When $k = 0$ it is straightforward to prove that
\[
    f_0^{(d)}(y) = \frac{(2d)!}{d!(1-4y)^{d+1/2}},
\]
so $f^{(d)}_0(0) = (2d)!/d! = \binom{2d}{d}\binom{d}{0}d!$. For $k > 0$ an explicit formula for $f_k^{(d)}(y)$ is much uglier, but it is still straightforward to prove (by induction on $d$, for example) that $f^{(d)}_k(0) = \binom{2d}{d}\binom{d}{k}d!$. Alternatively, this generating function is described in OEIS sequence A046521 \cite{oeisA046521}.

\textbf{Claim 2.} For each integer $s \geq 0$, the $s$-th derivative of the function $g(y) = 1/(1-4y)$ is
\[
    g^{(s)}(y) = \frac{s!4^s}{(1-4y)^{s+1}},
\]
which has Taylor series
\[
    g^{(s)}(y) = \sum_{d=s}^\infty 4^d s!\binom{d}{s} y^{d-s}.
\]

\emph{Proof of Claim~2.} The explicit formula for $g^{(s)}(y)$ is straightforward to prove by induction on $s$. The Taylor series can be obtained by repeatedly differentiating the geometric series $g(y) = \sum_{d=0}^\infty (4y)^d$.

\textbf{Claim 3.} If $s \geq k \geq 0$ and $f_k$ is as in Claim~1, the Taylor series (centered at $0$) of $f_k(y)f_{s-k}(y)$ is
\[
    f_k(y)f_{s-k}(y) = \binom{2k}{k}\binom{2s-2k}{s-k} \sum_{d=s}^\infty 4^{d-s} \binom{d}{s} y^d.
\]

\emph{Proof of Claim~3.} If $f_k$ is as in Claim~1 and $g$ is as in Claim~2 then direct computation shows that
\ba
    f_k(y)f_{s-k}(y) & = \left(\frac{\binom{2k}{k}y^k}{(1-4y)^{k + 1/2}}\right)\left(\frac{\binom{2s-2k}{s-k}y^{s-k}}{(1-4y)^{s-k + 1/2}}\right) \\
    & = \binom{2k}{k}\binom{2s-2k}{s-k}\frac{y^s}{(1-4y)^{s+1}} \\
    & = \binom{2k}{k}\binom{2s-2k}{s-k}\frac{y^s}{s! 4^s} g^{(s)}(y) \\
    & = \binom{2k}{k}\binom{2s-2k}{s-k}\frac{y^s}{s! 4^s} \left(\sum_{d=s}^\infty 4^d s!\binom{d}{s} y^{d-s}\right) \\
    & = \binom{2k}{k}\binom{2s-2k}{s-k} \sum_{d=s}^\infty 4^{d-s} \binom{d}{s} y^{d},
\ea
where the third and fourth equalities come from Claim~2.

\textbf{Claim 4.} If $s \geq k \geq 0$ and $d \geq 0$ is an integer then
\[
    \sum_{j=0}^d \binom{2j}{j}\binom{j}{k}\binom{2d-2j}{d-j}\binom{d-j}{s-k} = 4^{d-s} \binom{2k}{k}\binom{2s-2k}{s-k} \binom{d}{s}.
\]

\emph{Proof of Claim~4.} By Claim~3, the coefficient of $y^d$ in $f_k(y)f_{s-k}(y)$ is $4^{d-s} \binom{2k}{k}\binom{2s-2k}{s-k} \binom{d}{s}$. By Claim~1, the coefficient of $y^d$ in $f_k(y)f_{s-k}(y)$ is $\sum_{j=0}^d \binom{2j}{j}\binom{j}{k}\binom{2d-2j}{d-j}\binom{d-j}{s-k}$. These quantities must equal each other.

We are now done with the series of claims. Plugging the result of Claim~4 into Equation~\eqref{eq:xs_coeff_first} reveals that the coefficient of $x^s$ in Equation~\eqref{eq:delta_formula} is equal to
\ba\label{eq:xs_formula_coeff}
\frac{(-1)^s 2^{d-s} \binom{d}{s}}{\binom{2d}{d}} \sum_{k=0}^{s} (-1)^{k}\binom{2k}{k}\binom{2s-2k}{s-k}.
\ea
If $s$ is odd then this coefficient is equal to $0$, since for all $0 \leq k \leq s$ the $k$-th term in the sum~\eqref{eq:xs_formula_coeff} is equal to the negative of the $(s-k)$-th term. If $s$ is even then the alternating sum over $k$ in the sum~\eqref{eq:xs_formula_coeff} is equal to $2^s\binom{s}{s/2}$ (see \cite{centralksum}), which shows that the coefficient of $x^s$ in Equation~\eqref{eq:delta_formula} is equal to
\ba\label{eq:xs_formula_coeffb}
\frac{2^{d} \binom{d}{s}\binom{s}{s/2}}{\binom{2d}{d}}
\ea
in this case. In particular, this tells us that the coefficient of $x^s$ is non-negative when $s$ is even. It follows that $\delta$ is a sum of even powers of $x$, so it is minimized when $x = 0$ (i.e., when $t = 1/2$). In this case, $\delta={2^d}{\binom{2d}{d}^{-1}}$. This completes the proof.
\end{proof}

%


\bibliographystyle{alpha}
\bibliography{references}

\end{document}